\documentclass[11pt,reqno,twoside]{amsart}

\usepackage{mathrsfs,amsfonts,amssymb,amsmath,amsthm}
\usepackage{titletoc,geometry,bm,indentfirst,graphicx,stmaryrd,tikz}
\usepackage[normalem]{ulem}
\usepackage{color,soul,setspace,float,appendix,multicol,cite,enumitem}

\usepackage[colorlinks=true,citecolor=red,linkcolor=blue]{hyperref}
\newcommand{\refinblack}[2]{{\hypersetup{linkcolor=black}\hyperref[#1]{#2}}}

\newcommand{\Ttwo}{\mathbb{T}_{2}}
\newcommand{\Ttwopi}{\mathbb{T}_{2\pi}}
\newcommand{\R}{\mathbb{R}}

\newcommand{\M}{\mathcal{M}}

\newcommand{\Z}{\mathbb{Z}}
\newcommand{\dd}{\mathop{}\!\mathrm{d}}
\DeclareMathOperator{\meas}{meas}

\numberwithin{equation}{section}
\newtheorem{theorem}{Theorem}[section]
\newtheorem{lemma}[theorem]{Lemma}
\newtheorem{corollary}[theorem]{Corollary}
\newtheorem{proposition}[theorem]{Proposition}

\theoremstyle{definition}

\newtheorem{definition}[theorem]{Definition}

\title[Spacetime observable symmetry on interval]{\large Observable symmetry for spacetime observability of wave equations on the interval}

\author[Bin Huang]{Bin Huang}

\author[Shengquan Xiang]{Shengquan Xiang} 

\address[Bin Huang]{School of Mathematical Sciences, Peking University, 100871, Beijing, China.
}
\email{bhuang25@stu.pku.edu.cn}

\address[Shengquan  Xiang]{School of Mathematical Sciences, Peking University, 100871, Beijing, China.}
\email{shengquan.xiang@math.pku.edu.cn}

\begin{document}
\vspace{5mm}
\begin{abstract}
Following the work  \cite{NiuWangXiang2026}, we further study spacetime observability for the wave equation on a one-dimensional interval. The main task is to characterize the observable symmetry condition (OSC) under both Dirichlet and Neumann boundary conditions, which is more complex than the torus setting due to boundary reflection, and to discuss the relation between OSC under different settings. We prove that OSC, together with the geometric control condition (GCC), is necessary and sufficient for observability. We also provide a necessary and sufficient condition for the related unique continuation property. 
\end{abstract}
\maketitle

\section{Introduction}\label{sec:intro}
Let $T>0$, $I=(0,1)$, and let $G\subset[0,T]\times I$ be a measurable spacetime set. We consider the \textit{observability} of the wave equation on the observation set $G$: does there exist a constant $C>0$ such that every weak solution $u$ of
\vspace{1mm}

$(i)$ the Dirichlet boundary problem
\begin{equation}\label{eq:wave-di}\begin{cases}
\partial_t^2 u - \partial_x^2 u = 0, & (t,x) \in [0,T]\times I,\\
u(t,0) = u(t,1) = 0, & t \in [0,T],\\
(u,\partial_t u)|_{t=0} = (u_0,u_1) \in H_0^1(I) \times L^2(I)
\end{cases}\end{equation}

$(ii)$ or the Neumann boundary problem

\begin{equation}\label{eq:wave-ne}\begin{cases}
\partial_t^2 u - \partial_x^2 u = 0, & (t,x) \in [0,T]\times I,\\
\partial_x u(t,0) = \partial_x u(t,1) = 0, & t \in [0,T],\\
(u,\partial_t u)|_{t=0} = (u_0,u_1) \in H^1(I) \times L^2(I)
\end{cases}\end{equation}
satisfies the observability inequality
\begin{equation}\label{eq:obs} \|\partial_xu_0\|_{L^2(I)}^2+\|u_1\|_{L^2(I)}^2
\leq C\iint_G|\partial_tu(t,x)|^2\,\dd x\,\dd t.
\end{equation}  

\subsection{Background}

Through the Hilbert Uniqueness Method (HUM) introduced by Lions, exact controllability of the wave equation can be reduced to an observability inequality for the corresponding adjoint system; see \cite{Lions1988book,Coron2007book-control}. The central task is therefore to establish an energy estimate that reconstructs the full state from partial observations.

Several fundamental methods are available for deriving such inequalities for wave equations and, more generally, for hyperbolic systems. The Fourier method and characteristic method are particularly effective in one dimension, where spectral decompositions and Ingham-type inequalities yield sharp observability results and optimal control times; see \cite{ingham1936some,avdonin1995families, zuazua-review, Li-1, LY-2, Russell-1, DagerZuazua2006, Tucsnak2009book}. The multiplier method instead uses energy identities obtained by integration by parts with suitable first-order differential operators; see \cite{Lions1988book, miranda1995hum, Tucsnak2009book, Zhang2000}. Carleman estimates, which are weighted spacetime integral inequalities involving a large parameter, provide another robust method and are also a powerful tool for proving unique continuation; see \cite{carleman1939probleme,hormanderanalysis, DuyckaertsZhangZuazua2008, Tataru, Laurent-Leautaud-2019}.

\subsubsection{Geometric control condition}

Microlocal analysis led to a major advance in the control of wave equations: the geometric control condition (GCC), introduced by Bardos, Lebeau, and Rauch \cite{BLR-gcc}. Their seminal result states that, for a classical cylindrical observation region $G=(0,T)\times\omega$, the observability inequality holds precisely when every generalized bicharacteristic of the wave operator enters the observation region within the observation time. GCC reflects propagation of singularities and connects controllability with geometric optics. Its necessity is commonly demonstrated using Gaussian beams; see \cite{burq1997condition,ralston1982gaussian}.

Since then, GCC has become a standard setting for hyperbolic equations, with extensions to manifolds, variable-coefficient operators, and boundary control problems. It also plays a central role in stabilization and semiclassical analysis; see \cite{Burq-98, Dehman-Sylvain-Zuazua,Lebeau-1996, Krieger-Xiang-2024,zuazua-review, Coron-Krieger-Xiang-2025, Coron-Krieger-Xiang-GWM-2025,Dehman-Sylvain-Zuazua, Shao}.

\subsubsection{Observable symmetry condition}
More recently, observability theory has been extended beyond classical cylindrical regions to {\it general spacetime measurable sets} $G\subset (0, T)\times\Omega$. This setting is motivated by moving sensors, intermittent observations, and sparse measurement strategies; see \cite{Castro-Cindea-Munch-2014,RouLebeauAnalPDE2017,Shao,PK,NWX-KDV,burq-zhu}.  The observability on measurable sets for both parabolic and dispersive equations has been extensively investigated in the literature \cite{BM-23, WWXZ-2010, WWZZ, WangWangZhang2019, AEWZ, NWX-KDV, NiuWangXiang2026}.

In this direction \cite{NiuWangXiang2026}  discovered a new phenomenon: {\it although GCC remains necessary, it is no longer sufficient for general spacetime observation sets}. The work \cite{NiuWangXiang2026} demonstrates this for the wave equation on the one-dimensional torus. A measurable set $G\subset(0,T)\times\Ttwopi$ may satisfy GCC and even occupy almost all of spacetime, while observability still fails because wave components can cancel on $G$.
The spacetime {\it observable symmetry condition (OSC)}, introduced in \cite{NiuWangXiang2026}, rules out precisely the symmetric configurations of the observation set that permit this cancellation.

This motivates extending the analysis from the torus to domains with boundary. The one-dimensional interval with Dirichlet or Neumann boundary conditions is the basic model in which to examine how boundary reflection, symmetry, and interference interact with spacetime observation.

\subsection{Main results}
We extend the torus analysis to the interval with both Dirichlet and Neumann boundary conditions. In either case, GCC remains necessary but is not sufficient for observability. Figures~\ref{fig:intro-di-ctex} and~\ref{fig:intro-ne-tor-ctex} illustrate the distinct geometric mechanisms involved.

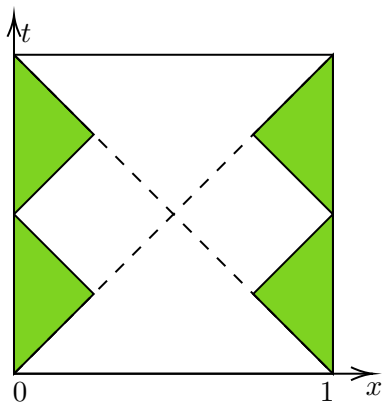
\begin{figure}[h]
\tikzset{every picture/.style={line width=0.75pt}} 
\begin{tikzpicture}[x=0.75pt,y=0.75pt,yscale=-1,xscale=1]
\draw   (40,40) -- (200,40) -- (200,200) -- (40,200) -- cycle ;
\draw    (40,200) -- (40,22) ;
\draw [shift={(40,20)}, rotate = 90] [color={rgb, 255:red, 0; green, 0; blue, 0 }  ][line width=0.75]    (10.93,-3.29) .. controls (6.95,-1.4) and (3.31,-0.3) .. (0,0) .. controls (3.31,0.3) and (6.95,1.4) .. (10.93,3.29)   ;
\draw    (40,200) -- (218,200) ;
\draw [shift={(220,200)}, rotate = 180] [color={rgb, 255:red, 0; green, 0; blue, 0 }  ][line width=0.75]    (10.93,-3.29) .. controls (6.95,-1.4) and (3.31,-0.3) .. (0,0) .. controls (3.31,0.3) and (6.95,1.4) .. (10.93,3.29)   ;
\draw  [fill={rgb, 255:red, 126; green, 211; blue, 33 }  ,fill opacity=1 ] (40,40) -- (80,80) -- (40,120) -- cycle ;
\draw  [fill={rgb, 255:red, 126; green, 211; blue, 33 }  ,fill opacity=1 ] (40,120) -- (80,160) -- (40,200) -- cycle ;
\draw  [fill={rgb, 255:red, 126; green, 211; blue, 33 }  ,fill opacity=1 ] (200,40) -- (200,120) -- (160,80) -- cycle ;
\draw  [fill={rgb, 255:red, 126; green, 211; blue, 33 }  ,fill opacity=1 ] (200,120) -- (200,200) -- (160,160) -- cycle ;
\draw  [dash pattern={on 4.5pt off 4.5pt}]  (40,40) -- (200,200) ;
\draw  [dash pattern={on 4.5pt off 4.5pt}]  (40,200) -- (200,40) ;

\draw (38,203) node [anchor=north west][inner sep=0.75pt]   [align=left] {$0$};
\draw (215,203) node [anchor=north west][inner sep=0.75pt]   [align=left] {$x$};
\draw (192,203) node [anchor=north west][inner sep=0.75pt]   [align=left] {$1$};
\draw (42,23) node [anchor=north west][inner sep=0.75pt]   [align=left] {$t$};
\draw (261,45) node [anchor=north west][inner sep=0.75pt]   [align=left] {The green region represents $G$.};

\draw (261,80) node [anchor=north west][inner sep=0.75pt]   [align=left] {There is a solution $u(t,x)$ of \eqref{eq:wave-di}};
\draw (261,105) node [anchor=north west][inner sep=0.75pt]   [align=left] {with $\partial_tu=0$ a.e. on $G$.};

\draw (261,140) node [anchor=north west][inner sep=0.75pt]   [align=left] {But $(u(0,x),\partial_tu(0,x))\neq(0,0)$,};

\draw (261,165) node [anchor=north west][inner sep=0.75pt]   [align=left] {$G$ satisfies GCC,};

\draw (261,190) node [anchor=north west][inner sep=0.75pt]   [align=left] {but observability \eqref{eq:obs} fails on $G$.};
\end{tikzpicture}
    \caption{Insufficiency of GCC in the Dirichlet setting}
    \label{fig:intro-di-ctex}
\end{figure}

\begin{figure}[h]
\tikzset{every picture/.style={line width=0.75pt}} 
\begin{tikzpicture}[x=0.75pt,y=0.75pt,yscale=-1,xscale=1]
\draw   (100,60) -- (260,60) -- (260,220) -- (100,220) -- cycle ;
\draw    (100,220) -- (100,42) ;
\draw [shift={(100,40)}, rotate = 90] [color={rgb, 255:red, 0; green, 0; blue, 0 }  ][line width=0.75]    (10.93,-3.29) .. controls (6.95,-1.4) and (3.31,-0.3) .. (0,0) .. controls (3.31,0.3) and (6.95,1.4) .. (10.93,3.29)   ;
\draw    (100,220) -- (278,220) ;
\draw [shift={(280,220)}, rotate = 180] [color={rgb, 255:red, 0; green, 0; blue, 0 }  ][line width=0.75]    (10.93,-3.29) .. controls (6.95,-1.4) and (3.31,-0.3) .. (0,0) .. controls (3.31,0.3) and (6.95,1.4) .. (10.93,3.29)   ;
\draw  [fill={rgb, 255:red, 245; green, 166; blue, 35 }  ,fill opacity=1 ] (100,60) -- (140,100) -- (180,60) -- cycle ;
\draw  [fill={rgb, 255:red, 245; green, 166; blue, 35 }  ,fill opacity=1 ] (180,220) -- (140,180) -- (100,220) -- cycle ;
\draw  [fill={rgb, 255:red, 245; green, 166; blue, 35 }  ,fill opacity=1 ] (260,60) -- (180,60) -- (220,100) -- cycle ;
\draw  [fill={rgb, 255:red, 245; green, 166; blue, 35 }  ,fill opacity=1 ] (180,220) -- (260,220) -- (220,180) -- cycle ;
\draw   (350,60) -- (510,60) -- (510,220) -- (350,220) -- cycle ;
\draw    (350,220) -- (350,42) ;
\draw [shift={(350,40)}, rotate = 90] [color={rgb, 255:red, 0; green, 0; blue, 0 }  ][line width=0.75]    (10.93,-3.29) .. controls (6.95,-1.4) and (3.31,-0.3) .. (0,0) .. controls (3.31,0.3) and (6.95,1.4) .. (10.93,3.29)   ;
\draw    (350,220) -- (528,220) ;
\draw [shift={(530,220)}, rotate = 180] [color={rgb, 255:red, 0; green, 0; blue, 0 }  ][line width=0.75]    (10.93,-3.29) .. controls (6.95,-1.4) and (3.31,-0.3) .. (0,0) .. controls (3.31,0.3) and (6.95,1.4) .. (10.93,3.29)   ;
\draw  [fill={rgb, 255:red, 208; green, 2; blue, 27 }  ,fill opacity=1 ] (430,60) -- (470,100) -- (430,140) -- (390,100) -- cycle ;
\draw  [color={rgb, 255:red, 0; green, 0; blue, 0 }  ,draw opacity=1 ][fill={rgb, 255:red, 208; green, 2; blue, 27 }  ,fill opacity=1 ] (430,140) -- (470,180) -- (430,220) -- (390,180) -- cycle ;

\draw (95,221) node [anchor=north west][inner sep=0.75pt]   [align=left] {$0$};
\draw (50,250) node [anchor=north west][inner sep=0.75pt]   [align=left] {Neumann case: $G$ is the orange region.};
\draw (340,250) node [anchor=north west][inner sep=0.75pt]   [align=left] {Torus case: $G$ is the red region.};
\draw (345,221) node [anchor=north west][inner sep=0.75pt]   [align=left] {$0$};
\draw (284,210) node [anchor=north west][inner sep=0.75pt]   [align=left] {$x$};
\draw (534,210) node [anchor=north west][inner sep=0.75pt]   [align=left] {$x$};
\draw (102,43) node [anchor=north west][inner sep=0.75pt]   [align=left] {$t$};
\draw (352,43) node [anchor=north west][inner sep=0.75pt]   [align=left] {$t$};
\draw (253,221) node [anchor=north west][inner sep=0.75pt]   [align=left] {$1$};
\draw (505,221) node [anchor=north west][inner sep=0.75pt]   [align=left] {$1$};
\end{tikzpicture}
    \caption{Insufficiency of GCC in the Neumann and torus settings}
    \label{fig:intro-ne-tor-ctex}
\end{figure}

This raises the central question: which geometric condition on $G$ precisely characterizes observability for the wave equation on an interval? Our main theorem provides the answer.

\begin{theorem}\label{thm:main-obs}
Let $T>0$ and let $G\subset[0,T]\times I$ be measurable. The following properties are equivalent.
\begin{enumerate}
    \item The observability inequality \eqref{eq:obs} holds for every solution $u$ to \eqref{eq:wave-di} (resp. \eqref{eq:wave-ne}).
    \item $G$ satisfies \refinblack{def:osc-d}{\textnormal{(OSC-D)}} (resp. \refinblack{def:osc-n}{\textnormal{(OSC-N)}}) and \refinblack{def:gcc}{\textnormal{(GCC)}}.
\end{enumerate}

\end{theorem}

Theorem~\ref{thm:main-obs} gives a complete characterization for an arbitrary measurable spacetime observation set on the interval, under either of the two standard boundary conditions. 
Note that the two interval symmetry conditions are genuinely different: neither (OSC-D) nor (OSC-N) implies the other. Nevertheless, they fit exactly into the torus theory after reflection. Let $G_{\mathrm{ref}}\subset[0,T]\times\Ttwo$ be the reflected periodization of $G$ defined in~\eqref{eq:g-ref}. For brevity, write (OBS-D), (OBS-N), and (OBS-T) for the validity of the observability inequality in the Dirichlet, Neumann, and period-$2$ torus settings, respectively. Section~\ref{subsec:diff-n} proves that reflection preserves (GCC), that (OSC-T) on $G_{\mathrm{ref}}$ is equivalent to the conjunction of (OSC-D) and (OSC-N) on $G$, and consequently that (OBS-T) on $G_{\mathrm{ref}}$ is equivalent to the simultaneous validity of (OBS-D) and (OBS-N) on $G$. 
\vspace{2mm}

The standard compactness--uniqueness argument combines a unique continuation property (UCP) with weak observability to obtain observability. Here UCP is the following qualitative analogue of observability:
\begin{equation*}\begin{aligned}
&\text{if $\partial_tu=0$ a.e. on $G$, then $u=0$ for the Dirichlet problem,}\\
&\text{and $u$ is constant for the Neumann problem.}
\end{aligned}\end{equation*}
Failure of UCP immediately implies failure of observability. The following characterization is also of independent interest.

\begin{theorem}\label{thm:main-ucp}
Let $T>0$ and let $G\subset[0,T]\times I$ be measurable. The following properties are equivalent.
\begin{enumerate}
    \item Unique continuation holds for every solution $u$ to \eqref{eq:wave-di} (resp. \eqref{eq:wave-ne}).
    \item $G$ satisfies \refinblack{def:osc-d}{\textnormal{(OSC-D)}} (resp. \refinblack{def:osc-n}{\textnormal{(OSC-N)}}) and \refinblack{def:w-gcc}{\textnormal{(Weak GCC)}}.
\end{enumerate}
\end{theorem}

Thus, the paper {\it focuses on formulating OSC for the Dirichlet and Neumann problems} in Sections~\ref{sec:osc-d} and~\ref{sec:osc-n}, and on relating these conditions to UCP in Sections~\ref{sec:ucp-d} and~\ref{sec:ucp-n}. Unlike on the torus, characteristic lines on a finite interval reflect at the boundary, so two lines initially moving in the same direction may intersect after reflection. For the quantities relevant to the Neumann problem, reflection also introduces an additional sign change.

To formulate OSC on the interval, we work on the labelled double cover $I\times\{\xi,\eta\}$. This gives a one-to-one correspondence between labelled starting points and the two families of characteristics, while keeping track of boundary reflections.

There is a further difference from the torus setting. On the torus, symmetry is described using the symmetric difference of two characteristic cylinders. On the interval, the labelled double cover reduces the condition to testing whether the two characteristics through a point of $G$ originate in the same measurable subset or in complementary subsets. See Figure~\ref{fig:intro-osc} and Sections~\ref{subsec:diff-d} and~\ref{subsec:diff-n}.

\begin{figure}[h!]  
\begin{tikzpicture}[x=0.75pt,y=0.75pt,yscale=-1,xscale=1]
\tikzset{every picture/.style={line width=0.75pt}} 
\draw  (80,320) -- (280,320)(99.87,135.83) -- (99.87,320) (273,315) -- (280,320) -- (273,325) (94.87,142.83) -- (99.87,135.83) -- (104.87,142.83)  ;
\draw  [fill={rgb, 255:red, 128; green, 128; blue, 128 }  ,fill opacity=1 ] (100,160) -- (260,160) -- (260,320) -- (100,320) -- cycle ;
\draw  (360,319.97) -- (560,319.97)(379.87,135.8) -- (379.87,319.97) (553,314.97) -- (560,319.97) -- (553,324.97) (374.87,142.8) -- (379.87,135.8) -- (384.87,142.8)  ;
\draw  [color={rgb, 255:red, 0; green, 0; blue, 0 }  ,draw opacity=1 ][fill={rgb, 255:red, 128; green, 128; blue, 128 }  ,fill opacity=1 ] (380,160) -- (540,160) -- (540,320) -- (380,320) -- cycle ;
\draw  [draw opacity=0][fill={rgb, 255:red, 255; green, 255; blue, 255 }  ,fill opacity=1 ] (260,210) -- (149.87,320) -- (99.87,320) -- (260,160) -- cycle ;
\draw  [draw opacity=0][fill={rgb, 255:red, 255; green, 255; blue, 255 }  ,fill opacity=1 ] (540,240) -- (460,320) -- (411.17,319.42) -- (540,190) -- cycle ;
\draw  [draw opacity=0][fill={rgb, 255:red, 255; green, 255; blue, 255 }  ,fill opacity=1 ] (510,160) -- (460,160) -- (540,240) -- (540,190) -- cycle ;
\draw  [fill={rgb, 255:red, 228; green, 228; blue, 228 }  ,fill opacity=1 ] (540,190) -- (540,240) -- (515,215) -- (522.29,207.71) -- cycle ;
\draw [color={rgb, 255:red, 208; green, 2; blue, 27 }  ,draw opacity=1]  [line width=2]  (530,220) -- (430,320) ;
\draw [color={rgb, 255:red, 208; green, 2; blue, 27 }  ,draw opacity=1 ] [line width=2]  (530,220) -- (540,230) ;
\draw [color={rgb, 255:red, 208; green, 2; blue, 27 }  ,draw opacity=1 ] [line width=2]  (540,230) -- (450,320) ;
\draw  [draw opacity=0][fill={rgb, 255:red, 255; green, 255; blue, 255 }  ,fill opacity=1 ] (100,160) -- (260,320) -- (210,320) -- (100,210) -- cycle ;
\draw  [fill={rgb, 255:red, 228; green, 228; blue, 228 }  ,fill opacity=1 ] (204.93,265) -- (180,290) -- (155,265) -- (179.93,240) -- cycle ;
\draw [color={rgb, 255:red, 126; green, 211; blue, 33 }  ,draw opacity=1 ][line width=2]   (179.97,265) -- (235,320) ;
\draw [color={rgb, 255:red, 74; green, 144; blue, 226 }  ,draw opacity=1 ][line width=2]   (179.97,265) -- (125,320) ;
\draw [color={rgb, 255:red, 208; green, 2; blue, 27 }  ,draw opacity=1 ][line width=1.5]    (410,320) -- (460,320) ;
\draw [color={rgb, 255:red, 74; green, 144; blue, 226 }  ,draw opacity=1 ][line width=1.5]    (99.87,320) -- (149.87,320) ;
\draw [color={rgb, 255:red, 126; green, 211; blue, 33 }  ,draw opacity=1 ][line width=1.5]    (210,320) -- (260,320) ;
\draw  [dash pattern={on 4.5pt off 4.5pt}]  (100,330) -- (260,330) ;
\draw  [dash pattern={on 4.5pt off 4.5pt}]  (100,340) -- (260,340) ;
\draw  [dash pattern={on 4.5pt off 4.5pt}]  (380,330) -- (540,330) ;
\draw  [dash pattern={on 4.5pt off 4.5pt}]  (380,340) -- (540,340) ;
\draw  [color={rgb, 255:red, 74; green, 144; blue, 226 }  ,draw opacity=1 ][fill={rgb, 255:red, 74; green, 144; blue, 226 }  ,fill opacity=1 ] (100,328) -- (150,328) -- (150,332) -- (100,332) -- cycle ;
\draw  [color={rgb, 255:red, 126; green, 211; blue, 33 }  ,draw opacity=1 ][fill={rgb, 255:red, 126; green, 211; blue, 33 }  ,fill opacity=1 ] (210,338) -- (260,338) -- (260,342) -- (210,342) -- cycle ;
\draw  [color={rgb, 255:red, 208; green, 2; blue, 27 }  ,draw opacity=1 ][fill={rgb, 255:red, 208; green, 2; blue, 27 }  ,fill opacity=1 ] (410,328) -- (460,328) -- (460,332) -- (410,332) -- cycle ;
\draw  [color={rgb, 255:red, 208; green, 2; blue, 27 }  ,draw opacity=1 ][fill={rgb, 255:red, 208; green, 2; blue, 27 }  ,fill opacity=1 ] (410,338) -- (460,338) -- (460,342) -- (410,342) -- cycle ;

\draw (60,345) node [anchor=north west][inner sep=0.75pt]   [align=left] {Characteristics starting in distinct sets $\widetilde E$ and $\widetilde F$;};
\draw (60,365) node [anchor=north west][inner sep=0.75pt]   [align=left] {$\widetilde E$ is the blue block and $\widetilde F$ is the green block.};
\draw (380,345) node [anchor=north west][inner sep=0.75pt]   [align=left] {Characteristics starting in the same set $\widetilde A$;};
\draw (380,365) node [anchor=north west][inner sep=0.75pt]   [align=left] {$\widetilde A$ is the union of the two red blocks.};
\end{tikzpicture}
    \caption{Characterization of the symmetry condition on a one-dimensional interval}
    \label{fig:intro-osc}
\end{figure}
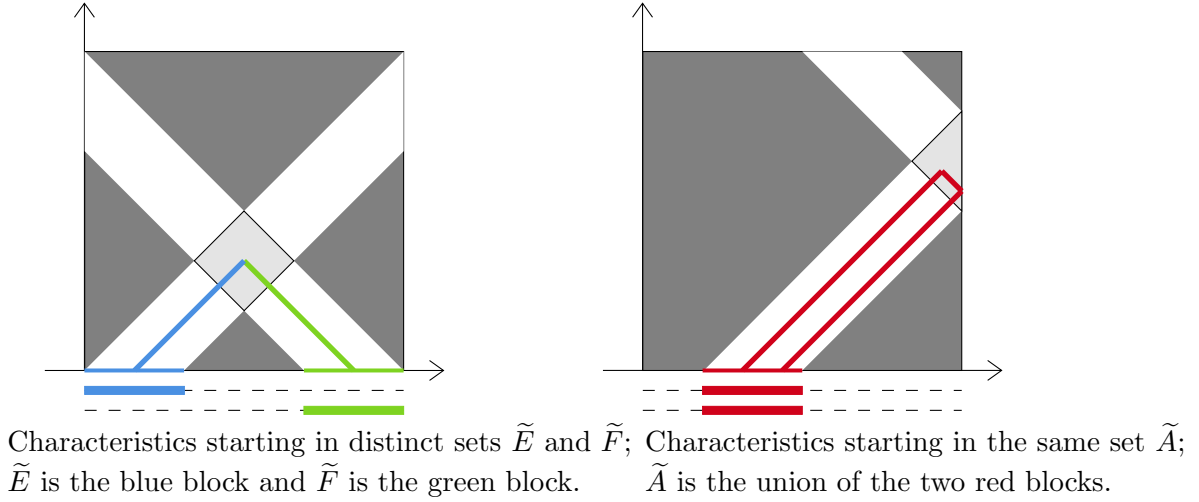

\subsection{Organization of this paper}
Section~\ref{sec:osc-d} introduces the characteristic coordinates and observable symmetry condition for the Dirichlet problem and compares them with the torus formulation. Section~\ref{sec:ucp-d} proves the corresponding UCP characterization.

Section~\ref{sec:osc-n} develops the reflection lemma and counterexamples for the Neumann problem. Its final subsection proves that (OSC-D) and (OSC-N) are independent and identifies their conjunction with (OSC-T) for the reflected periodization $G_{\mathrm{ref}}$; it also records the resulting equivalence among the three observability properties. The Neumann UCP argument is then given separately in Section~\ref{sec:ucp-n}.

Finally, Section~\ref{sec:obs} combines the UCP characterizations with weak observability and proves Theorems~\ref{thm:main-obs} and~\ref{thm:main-ucp}.

\section{Observable symmetry condition in the Dirichlet setting}\label{sec:osc-d}

We now turn to the observability problem in the Dirichlet boundary setting. In this section, we describe how the observable symmetry condition is formulated, analyze the differences from the torus case, and provide several illustrative examples.

\subsection{Basic definitions}\label{subsec:bas-def-d}
Let $u(t,x)\in C([0,T];H_0^1(I))\cap C^1([0,T];L^2(I))$
be a solution of~\eqref{eq:wave-di}. The \emph{characteristic lines} are given by $x\pm t=\text{constant}$ and reflect when they reach the boundary.

At each point $(t,x)\in[0,T]\times I$, there are two mutually orthogonal characteristic directions, denoted by $\xi$ and $\eta$, as illustrated in Figure~\ref{fig:char-lines}. We use the normalization
\begin{equation}\label{eq:char-der} 2\partial_\xi u=\partial_xu+\partial_tu,
\qquad
2\partial_\eta u=\partial_xu-\partial_tu.
\end{equation}

\begin{figure}[htp]
\tikzset{every picture/.style={line width=0.75pt}} 
\begin{tikzpicture}[x=0.75pt,y=0.75pt,yscale=100,xscale=100,scale=1.8]
\draw[->, thick] (0,0) -- (1.1,0) node[right] {$x$};
\draw[->, thick] (0,0) -- (0,0.9) node[above] {$t$};
\draw[->, thick] (0.5,0) -- (0.8,0.3) node[right] {$\xi$};
\draw[->, thick] (0.5,0) -- (0.3,0.2) node[above] {$\eta$};
\node[left] at (0,0.8) {$T$};
\node[below] at (1,0) {$1$};
\node[below left] at (0,0) {$0$};
\node[left] at (0.3,0.45) {$L_{[x_2;\eta]}$};
\node[right] at (0.7,0.55) {$L_{[x_1;\xi]}$};
\draw[thick] (0,0) -- (1,0) -- (1,0.8) -- (0,0.8) -- cycle;
\draw[red!60,thick] (0.75,0) -- (1,0.25) -- (0.45,0.8);
\node[below] at (0.75,0) {$x_1$};
\draw[blue!60,thick] (0.1,0) -- (0,0.1) -- (0.7,0.8);
\node[below] at (0.1,0) {$x_2$};
\node[right] at (1.2,0.4) {The characteristic line in blue is $L_{[x_2;\eta]}$.};
\node[right] at (1.2,0.6) {The characteristic line in red is $L_{[x_1;\xi]}$.}; 
\end{tikzpicture}
    \caption{Characteristic lines}
    \label{fig:char-lines}
\end{figure}

On the bottom interval $I$, the mark $\xi$ denotes characteristics initially directed to the right, whereas $\eta$ denotes those initially directed to the left. To distinguish the two families, define
\[ \widetilde I=I\times\{\xi,\eta\}, \]
endowed with the natural Lebesgue measure for which $|\widetilde I|=2$.

For $\widetilde x\in\widetilde I$, write
$\widetilde x=[x;\sigma]\; \forall x\in I,\forall  \sigma\in\{\xi,\eta\}$, 
and for $A\subset I$, we define
\[ [A;\sigma]=\{[x;\sigma]:x\in A\}. \]

Let $L_{[x;\xi]}$ be the characteristic line starting from $(0,x)$ in the $\xi$-direction, or equivalently from $[x;\xi]$. Define $L_{[x;\eta]}$ analogously. Figure~\ref{fig:char-lines} gives an example.

For each $(t,x)\in(0,T]\times(0,1)$, there is a unique ordered pair $(\widetilde y,\widetilde z)\in\widetilde I\times\widetilde I$ such that $(t,x)\in L_{\widetilde y}\cap L_{\widetilde z}$, with $L_{\widetilde y}$ oriented in the $\xi$-direction and $L_{\widetilde z}$ oriented in the $\eta$-direction at $(t,x)$.

\begin{definition}
For each $(t,x)\in(0,T]\times(0,1)$, the pair $(\widetilde y,\widetilde z)$ described above is called the \emph{generalized coordinates}, or simply the \emph{coordinates}, of $(t,x)$.
    
When $(t,x)$ lies on $([0,T]\times\{0,1\})\cup(\{0\}\times I)$, we set $\widetilde y=\widetilde z$.
\end{definition}

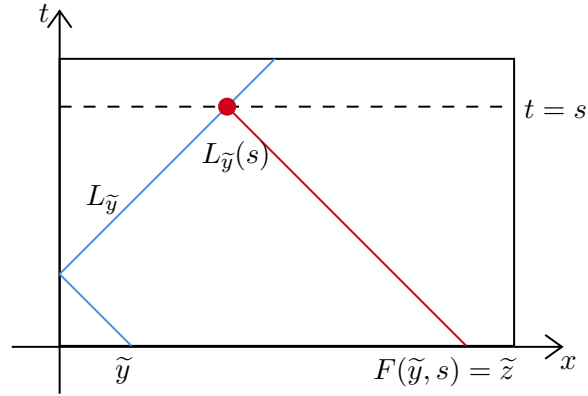
\begin{figure}[htp]
\tikzset{every picture/.style={line width=0.75pt}} 
\begin{tikzpicture}[x=0.75pt,y=0.75pt,yscale=-1,xscale=1,scale=1.2]
\draw  (80,280.4) -- (310,280.4)(100,140) -- (100,300) (303,275.4) -- (310,280.4) -- (303,285.4) (95,147) -- (100,140) -- (105,147)  ;
\draw   (100,160) -- (290,160) -- (290,280) -- (100,280) -- cycle ;
\draw [color={rgb, 255:red, 74; green, 144; blue, 226 }  ,draw opacity=1 ]   (100,250) -- (130,280) ;
\draw [color={rgb, 255:red, 74; green, 144; blue, 226 }  ,draw opacity=1 ]   (100,250) -- (190,160) ;
\draw  [dash pattern={on 4.5pt off 4.5pt}]  (100,180) -- (290,180) ;
\draw [color={rgb, 255:red, 208; green, 2; blue, 27 }  ,draw opacity=1 ]   (170,180) -- (270,280) ;
\draw [shift={(170,180)}, rotate = 45] [color={rgb, 255:red, 208; green, 2; blue, 27 }  ,draw opacity=1 ][fill={rgb, 255:red, 208; green, 2; blue, 27 }  ,fill opacity=1 ][line width=0.75]      (0, 0) circle [x radius= 3.35, y radius= 3.35]   ;

\draw (90,136) node [anchor=north west][inner sep=0.75pt]   [align=left] {$t$};
\draw (308,283) node [anchor=north west][inner sep=0.75pt]   [align=left] {$x$};
\draw (122,283.5) node [anchor=north west][inner sep=0.75pt]   [align=left] {$\widetilde{y}$};
\draw (293,176) node [anchor=north west][inner sep=0.75pt]   [align=left] {$t=s$};
\draw (230,283.5) node [anchor=north west][inner sep=0.75pt]   [align=left] {$F(\widetilde{y},s)=\widetilde{z}$};
\draw (109.67,211.5) node [anchor=north west][inner sep=0.75pt]   [align=left] {$L_{\widetilde{y}}$};
\draw (158,193) node [anchor=north west][inner sep=0.75pt]   [align=left] {$L_{\widetilde{y}}(s)$};
\end{tikzpicture}
    \caption{Definition of $F$}
    \label{fig:gen-coord-map}
\end{figure}

Given $(t,x)$ and one of its coordinates $\widetilde y$, the other coordinate $\widetilde z$ is uniquely determined. We therefore define a map
\[ F:\widetilde I\times[0,T]\longrightarrow\widetilde I, \qquad F(\widetilde y,t)=\widetilde z. \]

\subsection{Counterexample}\label{subsec:ctex-d}

We are now ready to give a counterexample in which $G$ satisfies (GCC), whereas the observability inequality~\eqref{eq:obs} fails. Here and below, $\chi_A$ denotes the indicator function of a measurable set $A$.

We set
\[ u\vert_{t=0} = x\chi_{(0,\frac{1}{2})}(x)+(1-x)\chi_{(\frac{1}{2},1)}(x),\quad\partial_t u\vert_{t=0}=0. \]
The construction of the counterexample relies on the following lemma.

\begin{lemma}\label{lem:refl-d}
Let $T>0$, and let $u$ be a solution of~\eqref{eq:wave-di}. There is a set $\mathcal N\subset[0,T]\times I$ with $\meas_{\R^2}(\mathcal N)=0$ such that, for every $(t,x)\in([0,T]\times I)\setminus\mathcal N$ with generalized coordinates $(\widetilde y,\widetilde z)$,
\[\partial_\eta u(t,x)=\left\{\begin{aligned}
    \partial_\eta u(0,y)&,\text{if } \widetilde y=[y;\xi],\\
    \partial_\xi u(0,y)&,\text{if } \widetilde y=[y;\eta],
\end{aligned}\right.\quad \partial_\xi u(t,x)=\left\{\begin{aligned}
    \partial_\xi u(0,z)&,\text{if } \widetilde z=[z;\eta],\\
    \partial_\eta u(0,z)&,\text{if } \widetilde z=[z;\xi].
\end{aligned}\right.
\]

\end{lemma}

The proof is elementary and is omitted. By Lemma~\ref{lem:refl-d},
\[\left.\partial_\xi u\right|_{G_0}=\left.\partial_\eta u\right|_{G_0}
=\left.\partial_\xi u\right|_{t=0,\,x\in(0,1/2)}=\left.\partial_\eta u\right|_{t=0,\,x\in(0,1/2)}
=\frac12,
\]
and
\[\left.\partial_\xi u\right|_{G_1}=\left.\partial_\eta u\right|_{G_1}
=\left.\partial_\xi u\right|_{t=0,\,x\in(1/2,1)}=\left.\partial_\eta u\right|_{t=0,\,x\in(1/2,1)}
=-\frac12.
\]
See Figure~\ref{fig:gcc-ctex-di}. Hence $\partial_tu=0$ almost everywhere on $G$, although $u$ is nonconstant.

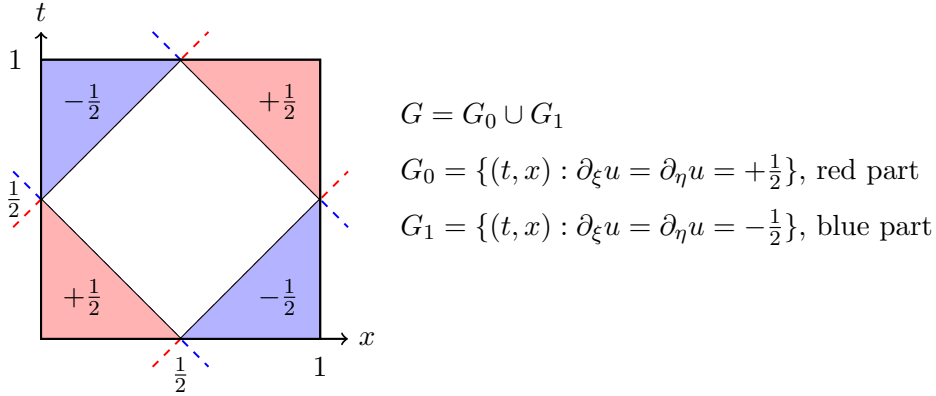
\begin{figure}[htp]
\tikzset{every picture/.style={line width=0.75pt}}
\begin{tikzpicture}[x=0.75pt,y=0.75pt,yscale=100,xscale=100,scale=1.4]
\draw[->, thick] (0,0) -- (1.1,0) node[right] {$x$};
\draw[->, thick] (0,0) -- (0,1.1) node[above] {$t$};

\node[left=0.1cm] at (0,1) {$1$};
\node[left=0.1cm] at (0,0.5) {$\frac{1}{2}$};
\node[below=0.1cm] at (1,0) {$1$};
\node[below=0.1cm] at (0.5,0) {$\frac{1}{2}$};

\draw[lightgray, very thin] (0,0) grid (1,1);

\draw[dashed, blue] (-0.1,0.6) -- (0.6,-0.1) ;
\draw[dashed, blue] (0.4,1.1) -- (1.1,0.4) ;
\draw[dashed, red] (0.4,-0.1) -- (1.1,0.6) ;
\draw[dashed, red] (-0.1,0.4) -- (0.6,1.1) ;

\draw[thick] (0,0)--(0.5,0)--(0,0.5) -- cycle;
\fill[red!30] (0,0)--(0.5,0)--(0,0.5) -- cycle;
\node at (0.15,0.15) {$+\frac12$};

\draw[thick] (0,0.5)--(0.5,1)--(0,1) -- cycle;
\fill[blue!30] (0,0.5)--(0.5,1)--(0,1) -- cycle;
\node at (0.15,0.85) {$-\frac12$};

\draw[thick] (1,0)--(0.5,0)--(1,0.5) -- cycle;
\fill[blue!30] (1,0)--(0.5,0)--(1,0.5) -- cycle;
\node at (0.85,0.15) {$-\frac12$};

\draw[thick] (1,0.5)--(0.5,1)--(1,1) -- cycle;
\fill[red!30] (1,0.5)--(0.5,1)--(1,1) -- cycle;
\node at (0.85,0.85) {$+\frac12$};

\draw[thick] (0,0) -- (1,0) -- (1,1) -- (0,1) -- cycle;

\node[right] at (1.25,0.8) {$G = G_0 \cup G_1$};
\node[right] at (1.25,0.6) {$G_0=\{(t,x):\partial_\xi u=\partial_\eta u=+\frac12\}$, red part};
\node[right] at (1.25,0.4) {$G_1=\{(t,x):\partial_\xi u=\partial_\eta u=-\frac12\}$, blue part};
\end{tikzpicture}
    \caption{Counterexample under Dirichlet boundary conditions}
    \label{fig:gcc-ctex-di}
\end{figure}

\subsection{New notations and definitions}\label{subsec:new-def-d}

We now explain the mechanism behind the counterexample. For $\widetilde A\subset\widetilde I$, define
\begin{align*} L_{\widetilde A}&=\bigcup_{\widetilde x\in\widetilde A}L_{\widetilde x},\\
\M_{\widetilde A}&=\bigl\{(t,x)\in[0,\infty)\times I:\text{there are distinct }\widetilde y,\widetilde z\in\widetilde A\\
&\hspace{42mm}\text{such that }(t,x)\in L_{\widetilde y}\cap L_{\widetilde z}\bigr\}.
\end{align*}

Thus, $L_{\widetilde A}$ is the union of the characteristic lines starting from $\widetilde A$, whereas $\M_{\widetilde A}$ consists of the points whose two generalized coordinates both belong to $\widetilde A$.
Define
\[ L_{[x;\xi]}(s)=L_{[x;\xi]}\cap\{t=s\}, \qquad L_{[x;\eta]}(s)=L_{[x;\eta]}\cap\{t=s\}. \]
For a set $\widetilde A\subset\widetilde I$, also set $L_{\widetilde A}(s)=L_{\widetilde A}\cap\{t=s\}$ and $\M_{\widetilde A}(\tau)=\M_{\widetilde A}\cap\{t\leq\tau\}$. We introduce the following geometric symmetry condition in our setting.
\begin{definition}
Let $\widetilde A\subset\widetilde I$ be measurable, and let $0\leq\tau\leq T$. A set $G\subset[0,T]\times I$ is said to be $(\widetilde A;\tau)$-observable symmetric if
    \begin{equation}\label{eq:obs-sym-di} \meas_{\R^2}\!\left(G\setminus\left(\M_{\widetilde A}(\tau)\cup
\M_{\widetilde I\setminus\widetilde A}(\tau)\right)\right)=0.
\end{equation} 
Equivalently, almost every point of $G$ belongs either to $\M_{\widetilde A}$ or to $\M_{\widetilde I\setminus\widetilde A}$.
\end{definition}

A set $\widetilde A$ is called trivial if $|\widetilde A|=0$ or $|\widetilde A|=2$; otherwise, it is called nontrivial.

\begin{definition}\label{def:gcc}
Let $T>0$. A measurable set $G\subset[0,T]\times I$ satisfies \emph{(GCC)} if there is a constant $c_0>0$ such that, for almost every $x\in[0,1]$,
\begin{align*} &\mathcal H^1\!\left(G\cap\{L_{[x;\eta]}(s):s\in[0,T]\}\right)\geq c_0,\\
&\mathcal H^1\!\left(G\cap\{L_{[x;\xi]}(s):s\in[0,T]\}\right)\geq c_0. \end{align*}
\end{definition}

\begin{definition}[Observable symmetry condition in the Dirichlet setting]\label{def:osc-d}
Let $T>0$. A measurable set $G\subset[0,T]\times I$ satisfies \emph{(OSC-D)} if, for every nontrivial measurable set $\widetilde A\subset\widetilde I$, the set $G$ is not $(\widetilde A;T)$-observable symmetric.
\end{definition}
A schematic illustration of (OSC-D) is shown in Figure~\ref{fig:osc-di}.

\begin{figure}[htp]
\tikzset{every picture/.style={line width=0.75pt}} 
\begin{tikzpicture}[x=0.75pt,y=0.75pt,yscale=-1,xscale=1,scale=1.4]
\draw  (80,320) -- (280,320)(99.87,160) -- (99.87,320) (273,315) -- (280,320) -- (273,325) (94.87,167) -- (99.87,160) -- (104.87,167)  ;
\draw  [fill={rgb, 255:red, 212; green, 212; blue, 212 }  ,fill opacity=1 ] (100,180) -- (260,180) -- (260,320) -- (100,320) -- cycle ;
\draw  [dash pattern={on 4.5pt off 4.5pt}]  (100,330) -- (260,330) ;
\draw  [dash pattern={on 4.5pt off 4.5pt}]  (100,340) -- (260,340) ;
\draw  [color={rgb, 255:red, 74; green, 144; blue, 226 }  ,draw opacity=1 ][fill={rgb, 255:red, 74; green, 144; blue, 226 }  ,fill opacity=1 ] (120,328) -- (140,328) -- (140,332) -- (120,332) -- cycle ;
\draw  [color={rgb, 255:red, 74; green, 144; blue, 226 }  ,draw opacity=1 ][fill={rgb, 255:red, 74; green, 144; blue, 226 }  ,fill opacity=1 ] (170,328) -- (210,328) -- (210,332) -- (170,332) -- cycle ;
\draw  [color={rgb, 255:red, 208; green, 2; blue, 27 }  ,draw opacity=1 ][fill={rgb, 255:red, 208; green, 2; blue, 27 }  ,fill opacity=1 ] (190,338) -- (230,338) -- (230,342) -- (190,342) -- cycle ;
\draw  [color={rgb, 255:red, 74; green, 144; blue, 226 }  ,draw opacity=1 ][fill={rgb, 255:red, 255; green, 255; blue, 255 }  ,fill opacity=1 ] (100,300) -- (100,280) -- (140,320) -- (120,320) -- cycle ;
\draw  [color={rgb, 255:red, 74; green, 144; blue, 226 }  ,draw opacity=1 ][fill={rgb, 255:red, 255; green, 255; blue, 255 }  ,fill opacity=1 ] (220,180) -- (100,300) -- (100,280) -- (200,180) -- cycle ;
\draw  [color={rgb, 255:red, 74; green, 144; blue, 226 }  ,draw opacity=1 ][fill={rgb, 255:red, 255; green, 255; blue, 255 }  ,fill opacity=1 ] (100,210) -- (210,320) -- (170,320) -- (100,250) -- cycle ;
\draw  [color={rgb, 255:red, 74; green, 144; blue, 226 }  ,draw opacity=1 ][fill={rgb, 255:red, 255; green, 255; blue, 255 }  ,fill opacity=1 ] (130,180) -- (170,180) -- (100,250) -- (100,210) -- cycle ;
\draw  [color={rgb, 255:red, 0; green, 0; blue, 0 }  ,draw opacity=1 ][fill={rgb, 255:red, 126; green, 211; blue, 33 }  ,fill opacity=1 ] (120,230) -- (100,250) -- (100,210) -- cycle ;
\draw  [color={rgb, 255:red, 208; green, 2; blue, 27 }  ,draw opacity=1 ][fill={rgb, 255:red, 255; green, 255; blue, 255 }  ,fill opacity=1 ] (260,250) -- (260,290) -- (230,320) -- (190,320) -- cycle ;
\draw  [color={rgb, 255:red, 208; green, 2; blue, 27 }  ,draw opacity=1 ][fill={rgb, 255:red, 255; green, 255; blue, 255 }  ,fill opacity=1 ] (190,180) -- (260,250) -- (260,290) -- (150,180) -- cycle ;
\draw  [color={rgb, 255:red, 0; green, 0; blue, 0 }  ,draw opacity=1 ][fill={rgb, 255:red, 126; green, 211; blue, 33 }  ,fill opacity=1 ] (110,290) -- (100,300) -- (100,280) -- cycle ;
\draw  [color={rgb, 255:red, 0; green, 0; blue, 0 }  ,draw opacity=1 ][fill={rgb, 255:red, 126; green, 211; blue, 33 }  ,fill opacity=1 ] (260,290) -- (260,250) -- (240,270) -- cycle ;
\draw  [color={rgb, 255:red, 0; green, 0; blue, 0 }  ,draw opacity=1 ][fill={rgb, 255:red, 65; green, 117; blue, 5 }  ,fill opacity=1 ] (170,180) -- (160,190) -- (150,180) -- cycle ;
\draw  [color={rgb, 255:red, 0; green, 0; blue, 0 }  ,draw opacity=1 ][fill={rgb, 255:red, 65; green, 117; blue, 5 }  ,fill opacity=1 ] (210,320) -- (200,310) -- (190,320) -- cycle ;
\draw  [color={rgb, 255:red, 0; green, 0; blue, 0 }  ,draw opacity=1 ][fill={rgb, 255:red, 65; green, 117; blue, 5 }  ,fill opacity=1 ] (145,255) -- (125,275) -- (115,265) -- (135,245) -- cycle ;
\draw  [color={rgb, 255:red, 0; green, 0; blue, 0 }  ,draw opacity=1 ][fill={rgb, 255:red, 65; green, 117; blue, 5 }  ,fill opacity=1 ][line width=0.75]  (205,195) -- (185,215) -- (175,205) -- (195,185) -- cycle ;

\draw (98,150) node [anchor=north west][inner sep=0.75pt]   [align=left] {$t$};
\draw (285,320) node [anchor=north west][inner sep=0.75pt]   [align=left] {$x$};
\draw (290,195) node [anchor=north west][inner sep=0.75pt]   [align=left] {$\widetilde{A}$ is the union of blue and red blocks.};
\draw (290,220) node [anchor=north west][inner sep=0.75pt]   [align=left] {$\M_{\widetilde{A}}(T)$ is the union of \\ dark green and light green parts.};
\draw (290,250) node [anchor=north west][inner sep=0.75pt]   [align=left] {$\M_{\widetilde{I}\backslash\widetilde{A}}(T)$ is the gray part.};
\draw (290,275) node [anchor=north west][inner sep=0.75pt]   [align=left] {If $G$ is a subset of $\M_{\widetilde{A}}(T)\cup \M_{\widetilde{I}\backslash\widetilde{A}}(T)$, \\ $G$ is $(\widetilde{A};T)$-observable symmetric.};
\draw (100,349) node [anchor=north west][inner sep=0.75pt]   [align=left] {The upper black dashed line segment represents $[I;\eta]$;};
\draw (100,364) node [anchor=north west][inner sep=0.75pt]   [align=left] {The lower black dashed line segment represents $[I;\xi]$.};
\end{tikzpicture}

    \caption{OSC in the Dirichlet setting}
    \label{fig:osc-di}
\end{figure}
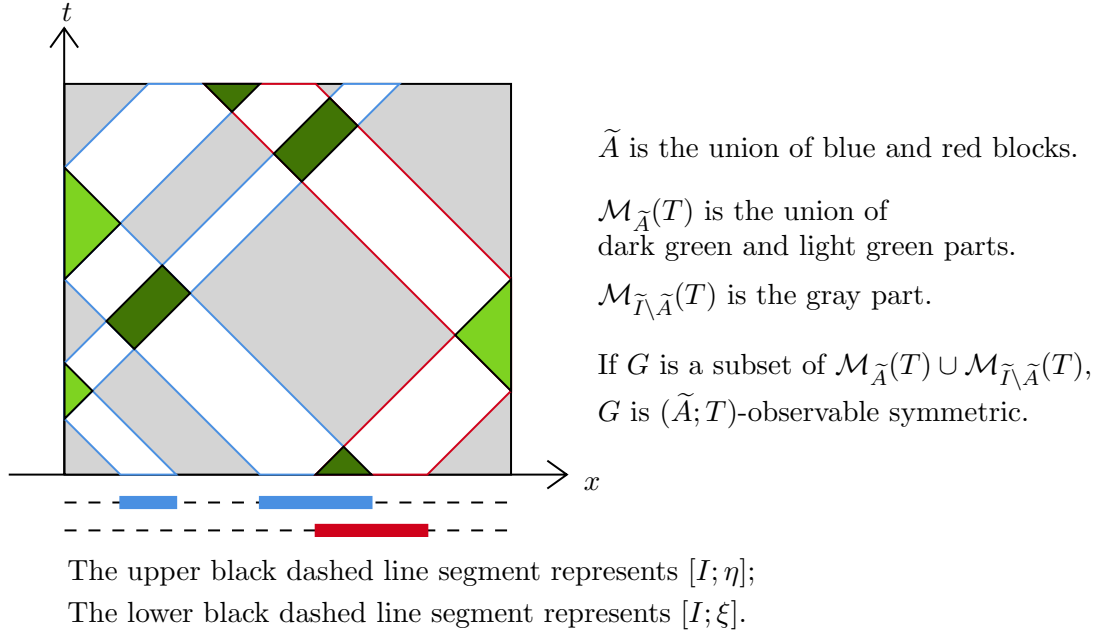

Returning to Section~\ref{subsec:ctex-d}, let
\[ \widetilde A=[I_0;\xi]\cup[I_0;\eta]\subset\widetilde I, \qquad I_0=\left[0,\frac12\right]. \]
Then $G$ is $(\widetilde A;1)$-observable symmetric: more precisely, $G_0\subset\M_{\widetilde A}(1)$ and $G_1\subset\M_{\widetilde I\setminus\widetilde A}(1)$. This explains why observability~\eqref{eq:obs} fails although $G$ satisfies (GCC).

Finally, we introduce the geometric assumption used in the discussion of UCP.
\begin{definition}\label{def:w-gcc}
Let $T>0$. A measurable set $G\subset[0,T]\times I$ satisfies \emph{(Weak GCC)} if, for almost every $x\in[0,1]$,
\begin{align*} &\meas_{\R}\left(G\cap\{L_{[x;\eta]}(s):s\in[0,T]\}\right)>0,\\
&\meas_{\R}\left(G\cap\{L_{[x;\xi]}(s):s\in[0,T]\}\right)>0. \end{align*}
\end{definition}

It follows directly from the definitions that (GCC) implies (Weak GCC).

\subsection{Comparison with the torus setting}
We now compare the observable symmetry conditions in the Dirichlet and torus settings.

\subsubsection{Observable symmetry condition on the torus}
We recall the observable symmetry condition on the torus from~\cite{NiuWangXiang2026}. Let $T>0$, let $\Ttwopi=\R/(2\pi\Z)$, and let $G\subset[0,T]\times\Ttwopi$ be measurable. Consider
\begin{equation}\label{eq:wave-tor-2pi} (\partial_t^2-\partial_x^2)v=0, \qquad (v,\partial_tv)|_{t=0}=(v_0,v_1)\in H^1(\Ttwopi)\times L^2(\Ttwopi).
\end{equation}

\begin{figure}[htp]
    \centering
\tikzset{every picture/.style={line width=0.75pt}} 
\begin{tikzpicture}[x=0.75pt,y=0.75pt,yscale=30,xscale=30,scale=0.9]
  \draw[->] (-0.2,0) -- (6.2,0) node[right] {$x$};
  \draw[->] (0,-0.2) -- (0,6.2) node[above] {$t$};
  \draw (3,0) node[below] {$\pi$} -- (3,0.1);
  \draw (6,0) node[below] {$2\pi$} -- (6,0.1);
  \draw (0,3) node[left] {$\pi$} -- (0.1,3);
  \draw (0,6) node[left] {$2\pi$} -- (0.1,6);
  \draw (0,0) -- (6,0) -- (6,6) -- (0,6) -- cycle;
\fill[blue!30] (0,0) coordinate (A1) -- (3,0) coordinate (B1) -- (1.5,1.5) coordinate (C1)-- cycle;
\node at (barycentric cs:A1=1,B1=1,C1=1) {\textcolor{white}{-1}};
\fill[red!30] (3,0) coordinate (A2) -- (6,0) coordinate (B2) -- (4.5,1.5) coordinate (C2)-- cycle;
\node at (barycentric cs:A2=1,B2=1,C2=1) {\textcolor{white}{1}};
\fill[red!30] (0,3) coordinate (A3) -- (1.5,1.5) coordinate (B3) -- (3,3) coordinate (C3)--(1.5,4.5) coordinate (D3)-- cycle;
\node at (barycentric cs:A3=1,B3=1,C3=1,D3=1) {\textcolor{white}{1}};
\fill[blue!30] (3,3) coordinate (A4) -- (4.5,1.5) coordinate (B4) -- (6,3) coordinate (C4)--(4.5,4.5) coordinate (D4)-- cycle;
\node at (barycentric cs:A4=1,B4=1,C4=1,D4=1) {\textcolor{white}{-1}};
\fill[blue!30] (0,6) coordinate (A5) -- (1.5,4.5) coordinate (B5) -- (3,6) coordinate (C5)-- cycle;
\node at (barycentric cs:A5=1,B5=1,C5=1) {\textcolor{white}{-1}};
\fill[red!30] (3,6) coordinate (A6) -- (4.5,4.5) coordinate (B6) -- (6,6) coordinate (C6)-- cycle;
\node at (barycentric cs:A6=1,B6=1,C6=1) {\textcolor{white}{1}};
  \draw[dashed] (-0.5,3.5) -- (3.5,-0.5); 
  \draw[dashed] (0,6) -- (6,0); 
  \draw[dashed] (2.5,6.5) -- (6.5,2.5); 
  \draw[dashed] (-0.5,2.5) -- (3.5,6.5); 
  \draw[dashed] (0,0) -- (6,6); 
  \draw[dashed] (2.5,-0.5) -- (6.5,3.5); 
\node at (14, 5) {$G=G_{1}\cup G_{2}$};  
\node at (14, 4) {$G_{1}=\{(t,x):\partial_\xi v=\partial_\eta v=-1\}$: blue part};
\node at (14, 3) {$G_2=\{(t,x):\partial_\xi v=\partial_\eta v=1\}$: red part};
\end{tikzpicture}
    \caption{Observability fails on $G$, although $G$ satisfies \textnormal{(GCC)}.}
    \label{fig:gcc-ctex-tor}
\end{figure}

For measurable sets $A,B\subset\Ttwopi$, define the characteristic cylinders
\begin{equation}\label{eq:cyl-tor} L_{\xi\in A}:=\bigcup_{\xi_0\in A}L_{\xi=\xi_0},
\qquad
L_{\eta\in B}:=\bigcup_{\eta_0\in B}L_{\eta=\eta_0}.
\end{equation}
Then $G\subset[0,T]\times\Ttwopi$ is said to be $(A,B)$-\emph{observable symmetric} if
\begin{equation}\label{eq:obs-sym-tor} \meas_{\R^2}\!\left([G\cap L_{\xi\in A}]\mathbin\Delta [G\cap L_{\eta\in B}]\right)=0,
\end{equation} 
where $X\mathbin\Delta Y=(X\setminus Y)\cup(Y\setminus X)$ is the symmetric difference.A pair $(A,B)$ is called trivial if either $|A|=|B|=0$ or $|A|=|B|=2\pi$; otherwise, it is called nontrivial.

\begin{definition}[Observable symmetry condition on the torus]\label{def:osc-t}
Let $T>0$. A set $G\subset[0,T]\times\Ttwopi$ satisfies \emph{(OSC-T)} if, for every nontrivial pair $(A,B)$, the set $G$ is not $(A,B)$-observable symmetric.
\end{definition}
In the counterexample shown in Figure~\ref{fig:gcc-ctex-tor}, the set $G$ is $(A,B)$-observable symmetric for $A=B=(0,\pi)$ and for $A=B=(\pi,2\pi)$.

\subsubsection{Geometric differences}\label{subsec:diff-d}
We can now compare the geometric conditions in the torus and Dirichlet settings. As Figure~\ref{fig:osc-di} shows, the light-green region is covered by an argument analogous to the torus case, whereas the dark-green region is not. The obstruction is the so-called \emph{boundary triangle}.

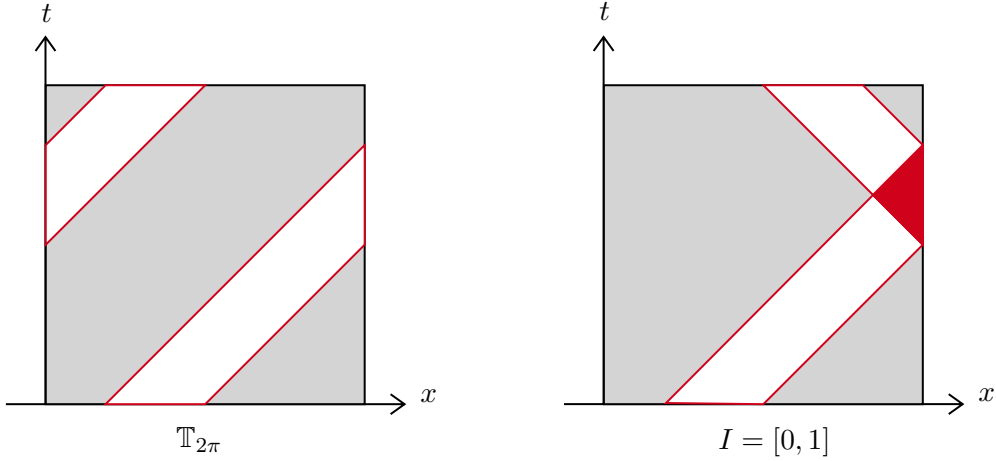
\begin{figure}[htp]
\tikzset{every picture/.style={line width=0.75pt}} 
\begin{tikzpicture}[x=0.75pt,y=0.75pt,yscale=-1,xscale=1,scale=1]

\draw  (80,320) -- (280,320)(99.87,135.83) -- (99.87,320) (273,315) -- (280,320) -- (273,325) (94.87,142.83) -- (99.87,135.83) -- (104.87,142.83)  ;
\draw  [fill={rgb, 255:red, 212; green, 212; blue, 212 }  ,fill opacity=1 ] (100,160) -- (260,160) -- (260,320) -- (100,320) -- cycle ;
\draw  (360,319.97) -- (560,319.97)(379.87,135.8) -- (379.87,319.97) (553,314.97) -- (560,319.97) -- (553,324.97) (374.87,142.8) -- (379.87,135.8) -- (384.87,142.8)  ;
\draw  [color={rgb, 255:red, 0; green, 0; blue, 0 }  ,draw opacity=1 ][fill={rgb, 255:red, 212; green, 212; blue, 212 }  ,fill opacity=1 ] (380,160) -- (540,160) -- (540,320) -- (380,320) -- cycle ;
\draw  [color={rgb, 255:red, 208; green, 2; blue, 27 }  ,draw opacity=1 ][fill={rgb, 255:red, 255; green, 255; blue, 255 }  ,fill opacity=1 ] (260,240) -- (180,320) -- (130,320) -- (260,190) -- cycle ;
\draw  [color={rgb, 255:red, 208; green, 2; blue, 27 }  ,draw opacity=1 ][fill={rgb, 255:red, 255; green, 255; blue, 255 }  ,fill opacity=1 ] (130,160) -- (100,190) -- (100,240) -- (180,160) -- cycle ;
\draw  [color={rgb, 255:red, 208; green, 2; blue, 27 }  ,draw opacity=1 ][fill={rgb, 255:red, 255; green, 255; blue, 255 }  ,fill opacity=1 ] (540,240) -- (460,320) -- (411.17,319.42) -- (540,190) -- cycle ;
\draw  [color={rgb, 255:red, 208; green, 2; blue, 27 }  ,draw opacity=1 ][fill={rgb, 255:red, 255; green, 255; blue, 255 }  ,fill opacity=1 ] (510,160) -- (460,160) -- (540,240) -- (540,190) -- cycle ;
\draw  [color={rgb, 255:red, 208; green, 2; blue, 27 }  ,draw opacity=1 ][fill={rgb, 255:red, 208; green, 2; blue, 27 }  ,fill opacity=1 ] (515,215) -- (540,240) -- (540,190) -- cycle ;

\draw (96.36,117.6) node [anchor=north west][inner sep=0.75pt]   [align=left] {$t$};
\draw (286.44,311.19) node [anchor=north west][inner sep=0.75pt]   [align=left] {$x$};
\draw (376.36,117) node [anchor=north west][inner sep=0.75pt]   [align=left] {$t$};
\draw (566.44,310.59) node [anchor=north west][inner sep=0.75pt]   [align=left] {$x$};
\draw (164,330) node [anchor=north west][inner sep=0.75pt]   [align=left] {$\Ttwopi$};
\draw (436,330) node [anchor=north west][inner sep=0.75pt]   [align=left] {$I=[0,1]$};
\end{tikzpicture}
    \caption{Boundary triangle}
    \label{fig:diff-tor-di}
\end{figure}

On the torus, intersections occur only between an $\eta$-characteristic and a $\xi$-characteristic. On an interval, reflections allow two characteristics that initially have the same orientation to intersect. Thus, when $G$ is contained in a boundary triangle, the observability inequality may fail.

\section{Unique continuation in the Dirichlet setting}\label{sec:ucp-d}
The unique continuation property (UCP) is central to the observability inequality~\eqref{eq:obs}. In this section, we study its relationship with \refinblack{def:gcc}{\textnormal{(GCC)}} and \refinblack{def:osc-d}{\textnormal{(OSC-D)}}. Following the strategy of~\cite{NiuWangXiang2026}, we prove the following implications:
\begin{itemize}
\item UCP implies Weak GCC (Proposition~\ref{prop:ucp-wgcc}).
\item UCP implies OSC-D (Proposition~\ref{prop:ucp-osc-d}).
\item Weak GCC and OSC-D imply UCP (Proposition~\ref{prop:ucp-d}).
\end{itemize}

The argument differs from the torus case in the construction of the counterexample and in the encoding of the initial characteristic derivatives $(\partial_\xi u,\partial_\eta u)|_{t=0}$ by a function $f$ on $\widetilde I$.

\begin{proposition}\label{prop:ucp-wgcc}
Let $T>0$, and let $G\subset[0,T]\times I$ be measurable. If UCP holds on $G$ for the Dirichlet problem~\eqref{eq:wave-di} (resp. the Neumann problem~\eqref{eq:wave-ne}), then $G$ satisfies \refinblack{def:w-gcc}{\textnormal{(Weak GCC)}}.
\end{proposition}

\begin{proof}
If Weak GCC fails, unfold the interval to $\Ttwo$ and choose a nonzero mean-zero characteristic datum supported on the positive-measure family of rays which misses $G$; its odd (resp. even) period-$2$ extension gives a nonzero (resp. nonconstant) solution with $\partial_tu=0$ almost everywhere on $G$, exactly as in Section~4.1, in particular Proposition~4.2, of~\cite{NiuWangXiang2026}.
\end{proof}

\begin{proposition}\label{prop:ucp-osc-d}
Let $T>0$, and let $G\subset[0,T]\times I$ be measurable. Assume that every solution $u$ of~\eqref{eq:wave-di} satisfies
\[ \partial_tu=0\quad\text{a.e. on }G \quad\Longrightarrow\quad u=0. \]
Then $G$ satisfies \refinblack{def:osc-d}{\textnormal{(OSC-D)}}.
\end{proposition}
\begin{proof}
We argue by contradiction. Suppose that, for some nontrivial set $\widetilde A$, the set $G$ is $(\widetilde A;T)$-observable symmetric. Then, as shown in Figure~\ref{fig:ucp-osc-di},
\begin{equation}\label{eq:ucp-sym-di} G\subset\M_{\widetilde A}(T)\cup
\M_{\widetilde I\setminus\widetilde A}(T)
\quad\text{up to a null set}.
\end{equation}
\definecolor{ffqqqq}{rgb}{1.,0.,0.}
\definecolor{wwccff}{rgb}{0.4,0.8,1.}
\definecolor{zzttqq}{rgb}{0.6,0.2,0.}
\begin{figure}[htp]
\tikzset{every picture/.style={line width=0.75pt}} 
\begin{tikzpicture}[x=0.75pt,y=0.75pt,yscale=-1,xscale=1,scale=1.4]

\draw  [fill={rgb, 255:red, 74; green, 144; blue, 226 }  ,fill opacity=1 ] (140,100) -- (340,100) -- (340,200) -- (140,200) -- cycle ;
\draw  [fill={rgb, 255:red, 255; green, 255; blue, 255 }  ,fill opacity=1 ][dash pattern={on 3.75pt off 3pt}] (300,100) -- (200,200) -- (180,200) -- (280,100) -- cycle ;
\draw  [fill={rgb, 255:red, 255; green, 255; blue, 255 }  ,fill opacity=1 ][dash pattern={on 3.75pt off 3pt}] (220,100) -- (320,200) -- (300,200) -- (200,100) -- cycle ;
\draw  [fill={rgb, 255:red, 255; green, 255; blue, 255 }  ,fill opacity=1 ][dash pattern={on 3.75pt off 3pt}] (340,180) -- (320,200) -- (300,200) -- (340,160) -- cycle ;
\draw  [fill={rgb, 255:red, 255; green, 255; blue, 255 }  ,fill opacity=1 ][dash pattern={on 3.75pt off 3pt}] (340,160) -- (280,100) -- (260,100) -- (340,180) -- cycle ;
\draw  [fill={rgb, 255:red, 208; green, 2; blue, 27 }  ,fill opacity=1 ][dash pattern={on 3.75pt off 3pt on 7.5pt off 1.5pt}] (260,140) -- (250,150) -- (240,140) -- (250,130) -- cycle ;
\draw  [fill={rgb, 255:red, 208; green, 2; blue, 27 }  ,fill opacity=1 ][dash pattern={on 3.75pt off 3pt}] (290,110) -- (280,120) -- (270,110) -- (280,100) -- cycle ;
\draw  [fill={rgb, 255:red, 208; green, 2; blue, 27 }  ,fill opacity=1 ][dash pattern={on 3.75pt off 3pt}] (310,190) -- (320,200) -- (300,200) -- cycle ;
\draw  [fill={rgb, 255:red, 208; green, 2; blue, 27 }  ,fill opacity=1 ][dash pattern={on 3.75pt off 3pt}] (340,180) -- (330,170) -- (340,160) -- cycle ;
\draw  [dash pattern={on 4.5pt off 4.5pt}]  (140,210) -- (340,210) ;
\draw  [dash pattern={on 4.5pt off 4.5pt}]  (140,220) -- (340,220) ;
\draw  [color={rgb, 255:red, 74; green, 144; blue, 226 }  ,draw opacity=1 ][fill={rgb, 255:red, 74; green, 144; blue, 226 }  ,fill opacity=1 ] (300,208) -- (320,208) -- (320,212) -- (300,212) -- cycle ;
\draw  [color={rgb, 255:red, 208; green, 2; blue, 27 }  ,draw opacity=1 ][fill={rgb, 255:red, 208; green, 2; blue, 27 }  ,fill opacity=1 ] (300,218) -- (320,218) -- (320,222) -- (300,222) -- cycle ;
\draw  [color={rgb, 255:red, 208; green, 2; blue, 27 }  ,draw opacity=1 ][fill={rgb, 255:red, 208; green, 2; blue, 27 }  ,fill opacity=1 ] (180,217.99) -- (200,217.99) -- (200,221.99) -- (180,221.99) -- cycle ;

\draw (135,230) node [anchor=north west][inner sep=0.75pt]   [align=left] {The upper black dashed line segment  represents $[I;\eta]$,};
\draw (135,245) node [anchor=north west][inner sep=0.75pt]   [align=left] {The lower black dashed line segment represents $[I;\xi]$.};

\end{tikzpicture}
    \caption{The set $G$ lies in the blue and red regions, while $\widetilde A$ is the union of the blue and red blocks.}
    \label{fig:ucp-osc-di}
\end{figure}

We construct a nonconstant solution such that $\partial_tu=0$ almost everywhere on $G$. Define
\[f=\chi_{\widetilde A}+c_A\chi_{\widetilde I\setminus\widetilde A}, \qquad
c_A=-\frac{\meas_{\R}(\widetilde A)}{2-\meas_{\R}(\widetilde A)}.
\]
Choose initial data $(u_0,u_1)\in H_0^1(I)\times L^2(I)$ such that

\begin{align}\label{eq:init-cor-di} \left.\partial_\xi u(x)\right|_{t=0}=f([x;\eta]),
\qquad
\left.\partial_\eta u(x)\right|_{t=0}=f([x;\xi]).
\end{align}
This choice verifies that
\[\int_I\partial_xu_0\,\dd x=\int_I(\partial_\xi u+\partial_\eta u)|_{t=0}\,\dd x
=\int_{\widetilde I}f(\widetilde x)\,\dd\widetilde x=0,
\]
so $u_0$ can be chosen with $u_0(0)=u_0(1)=0$. Lemma~\ref{lem:refl-d} gives
\[\partial_\eta u=\partial_\xi u=1\quad\text{on }\M_{\widetilde A}, \qquad
\partial_\eta u=\partial_\xi u=c_A\quad\text{on }\M_{\widetilde I\setminus\widetilde A}.
\]
Thus, $u$ is nonconstant and
\[ \partial_tu=\partial_\xi u-\partial_\eta u=0 \quad\text{on } \M_{\widetilde A}\cup\M_{\widetilde I\setminus\widetilde A}. \]
In view of~\eqref{eq:ucp-sym-di}, this contradicts UCP on $G$.
\end{proof}

Given a function $f$ on $\widetilde I$, define a pair $(g,h)$ on $I$ by
\[ g(x)=f([x;\eta]), \qquad h(x)=f([x;\xi]), \qquad x\in I. \]
Conversely, these identities determine a function $f$ on $\widetilde I$ from a pair $(g,h)$. We say that $f$ \emph{corresponds to} $(g,h)$ in the Dirichlet setting.

\begin{proposition}\label{prop:ucp-d}
Let $T>0$, and let $G\subset[0,T]\times I$ satisfy \refinblack{def:w-gcc}{\textnormal{(Weak GCC)}} and \refinblack{def:osc-d}{\textnormal{(OSC-D)}}. If $u$ solves~\eqref{eq:wave-di} and $\partial_tu=0$ almost everywhere on $G$, then $u=0$.
\end{proposition}
\begin{proof}
First, we prove that the function $f$ corresponding to $(\partial_\xi u,\partial_\eta u)|_{t=0}$ in the Dirichlet setting must satisfy
    \begin{equation}\label{eq:coun-sim-di} f(\widetilde x)=\sum_{k\in\Lambda}s_k\chi_{\widetilde A_k}(\widetilde x)
\quad\text{for a.e. }\widetilde x\in\widetilde I,
    \end{equation}
where the numbers $s_k$ are pairwise distinct, the sets $\widetilde A_k$ are pairwise disjoint, $\bigcup_{k\in\Lambda}\widetilde A_k=\widetilde I$, and $\Lambda$ is at most countable.

Since $G$ satisfies (Weak GCC), there is a set $\mathcal N_1\subset G$ such that $\meas_{\R^2}(\mathcal N_1)=0$ and $\partial_tu=0$ on $G\setminus\mathcal N_1$. Let $\mathcal N_2$ be the null set in Lemma~\ref{lem:refl-d}, and set $\mathcal N=\mathcal N_1\cup\mathcal N_2$. For $\widetilde x\in\widetilde I$, define
\[ J_{\widetilde x}=\{s\in[0,T]:L_{\widetilde x}(s)\in G\setminus\mathcal N\}. \]

By \refinblack{def:w-gcc}{\textnormal{(Weak GCC)}} and Fubini's theorem, for almost every $\widetilde x\in\widetilde I$, one has  $\meas_\R(J_{\widetilde x})>0$. Denote the set of such $\widetilde x$ by $\mathcal P$. For $\widetilde x\in\mathcal P$, let
\[ E_{\widetilde x}=\{\widetilde y\in\widetilde I: \widetilde y=F(\widetilde x,s)\text{ for some }s\in J_{\widetilde x}\}. \]
Then $E_{\widetilde x}\subset f^{-1}(f(\widetilde x))$ by Lemma~\ref{lem:refl-d}.

For $\widetilde x\in\mathcal P$ and $\widetilde y\in E_{\widetilde x}$, there are at most $N_T=\lfloor T\rfloor+2$ values of $s\in J_{\widetilde x}$ such that $F(\widetilde x,s)=\widetilde y$. Hence, for almost every $\widetilde x\in\widetilde I$,
\[ \meas_\R\!\left(f^{-1}(f(\widetilde x))\right) \geq\meas_\R(E_{\widetilde x}) \geq\frac1{N_T}\meas_\R(J_{\widetilde x}) >0. \]

Define
\[ \mathcal S=\{a\in\R:\meas_\R(f^{-1}(a))>0\}. \]
The set $\mathcal S$ is at most countable. Write $\mathcal S=\{s_k\}_{k\in\Lambda}$ with the $s_k$ pairwise distinct, and set $\widetilde A_k=f^{-1}(s_k)$. Then $|\widetilde A_k|>0$ for every $k\in\Lambda$, and $f$ has the form~\eqref{eq:coun-sim-di}.
    
We now show that $\#\Lambda=1$. Suppose instead that $\#\Lambda\geq2$, and let
\[ \mathcal Z=\{(t,x)\in[0,T]\times I:\partial_tu(t,x)=0\}. \]
Then $G\subset\mathcal Z$ up to a null set. For every $(t,x)\in\mathcal Z$, its generalized coordinates $(\widetilde y,\widetilde z)$ belong to the same $\widetilde A_k$. Consequently,
\[ \mathcal Z\subset\bigcup_{k\in\Lambda}\M_{\widetilde A_k}(T), \]
and therefore
\[G\subset\bigcup_{k\in\Lambda}\M_{\widetilde A_k}(T)
\subset\M_{\widetilde A_1}(T)\cup\M_{\widetilde I\setminus\widetilde A_1}(T)\quad\text{up to a null set}.
\]
This contradicts (OSC-D). Hence $\#\Lambda=1$, so $f=s_1$ almost everywhere on $\widetilde I$. The boundary conditions imply $s_1=0$, and consequently $u=0$.
\end{proof}

\section{Observable symmetry condition in the Neumann setting}\label{sec:osc-n}

We now turn to the Neumann boundary setting. The key difference from the Dirichlet setting is the sign change of the characteristic derivatives upon reflection at the boundary, as described in Lemma~\ref{lem:refl-n}.

\subsection{Counterexample}\label{subsec:ctex-n}
Let
\[ u(t,x)\in C([0,T];H^1(I))\cap C^1([0,T];L^2(I)) \]
be a solution of~\eqref{eq:wave-ne}. The reflection law now takes the following form.
\begin{lemma}\label{lem:refl-n}
Let $T>0$, and let $u$ be a solution of~\eqref{eq:wave-ne}. For almost every $(t,x)\in[0,T]\times I$, if the generalized coordinates of $(t,x)$ are $(\widetilde y,\widetilde z)$, then
\[\partial_\eta u(t,x)=\left\{\begin{aligned}
    \partial_\eta u(0,y)&,\text{if } \widetilde y=[y;\xi],\\
    -\partial_\xi u(0,y)&,\text{if } \widetilde y=[y;\eta],
\end{aligned}\right.\quad \partial_\xi u(t,x)=\left\{\begin{aligned}
    \partial_\xi u(0,z)&,\text{if } \widetilde z=[z;\eta],\\
    -\partial_\eta u(0,z)&,\text{if } \widetilde z=[z;\xi].
\end{aligned}\right.
\]
\end{lemma}
We can now give a counterexample in the Neumann setting. Set
\[u\vert_{t=0}=-\frac{1}{2}\chi_{(0,\frac{1}{4})}(x)+\frac{1}{2}\chi_{(\frac{3}{4},1)}(x)+2\left(x-\frac{1}{2}\right)\chi_{(\frac{1}{4},\frac{3}{4})}(x),
\]
and
\[ \partial_t u\vert_{t=0} = -2\chi_{(0,\frac{1}{4})}(x)+2\chi_{(\frac{3}{4},1)}(x), \]
Then
\[\left.\partial_\xi u\right|_{G_0}=\left.\partial_\eta u\right|_{G_0}
=\left.\partial_\xi u\right|_{\{0\}\times(\frac14,1)}
=\left.\partial_\eta u\right|_{\{0\}\times(0,\frac34)}=1,
\]
\[\left.\partial_\xi u\right|_{G_1}=\left.\partial_\eta u\right|_{G_1}
=\left.\partial_\xi u\right|_{\{0\}\times(0,\frac14)}
=\left.\partial_\eta u\right|_{\{0\}\times(\frac34,1)}=-1.
\]
See Figure~\ref{fig:gcc-ctex-ne}. Thus, $\partial_tu=0$ almost everywhere on a set $G$ satisfying (GCC), although $u$ is nonconstant.

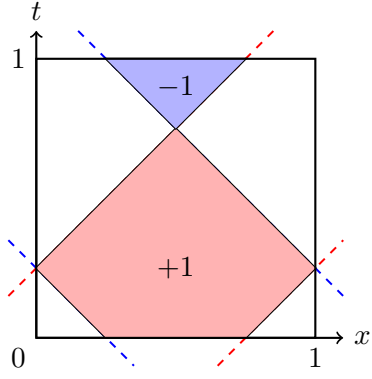
\begin{figure}[htp]
\tikzset{every picture/.style={line width=0.75pt}}
\begin{tikzpicture}[x=0.75pt,y=0.75pt,yscale=100,xscale=100,scale=1.4]
\draw[->, thick] (0,0) -- (1.1,0) node[right] {$x$};
\draw[->, thick] (0,0) -- (0,1.1) node[above] {$t$};
\node[left] at (0,1) {$1$};
\node[below] at (1,0) {$1$};
\node[below left] at (0,0) {$0$};
\draw[lightgray, very thin] (0,0) grid (1,1);
\draw[dashed, blue] (-0.1,0.35) -- (0.35,-0.1) ;
\draw[dashed, blue] (0.15,1.1) -- (1.1,0.15) ;
\draw[dashed, red] (-0.1,0.15) -- (0.85,1.1) ;
\draw[dashed, red] (0.65,-0.1) -- (1.1,0.35) ;
\draw[thick] (0.25,0) -- (0.75,0) -- (1,0.25) -- (0.5,0.75) -- (0,0.25) -- cycle;
\fill[red!30] (0.25,0) -- (0.75,0) -- (1,0.25) -- (0.5,0.75) -- (0,0.25) -- cycle;
\node at (0.5,0.25) {$+1$};

\draw[thick] (0.5,0.75) -- (0.75,1) -- (0.25,1) -- cycle;
\fill[blue!30](0.5,0.75) -- (0.75,1) -- (0.25,1) -- cycle;
\node at (0.5,0.9) {$-1$};
\draw[thick] (0,0) -- (1,0) -- (1,1) -- (0,1) -- cycle;
\node[right] at (1.25,0.8) {$G = G_0 \cup G_1$};
\node[right] at (1.25,0.6) {$G_0=\{(t,x):\partial_\xi u=\partial_\eta u=+1\}$, red part};
\node[right] at (1.25,0.4) {$G_1=\{(t,x):\partial_\xi u=\partial_\eta u=-1\}$, blue part};
\end{tikzpicture}
    \caption{Counterexample under Neumann boundary conditions}
    \label{fig:gcc-ctex-ne}
\end{figure}

\subsection{Definitions and notation}
To explain the preceding example, we use the notation introduced in Section~\ref{sec:osc-d} and add one new family of sets.

For two disjoint subsets $\widetilde E,\widetilde F\subset\widetilde I$, define
\[\mathcal W_{\widetilde E,\widetilde F}
=\{(t,x)\in[0,\infty)\times I:\text{there exist }\widetilde y\in\widetilde E,\ \widetilde z\in\widetilde F,
\ \widetilde y\neq\widetilde z,\text{ such that }(t,x)\in L_{\widetilde y}\cap L_{\widetilde z}\}.
\]
Set $\mathcal W_{\widetilde E,\widetilde F}(\tau)
=\mathcal W_{\widetilde E,\widetilde F}\cap\{t\leq\tau\}$.
\begin{definition}
Let $\widetilde E,\widetilde F\subset\widetilde I$ be disjoint measurable subsets, and let $0\leq\tau\leq T$. A set $G\subset[0,T]\times I$ is said to be $(\widetilde E,\widetilde F;\tau)$-observable symmetric if
    \begin{equation}\label{eq:obs-sym-ne} \meas_{\R^2}\!\left(G\setminus\left(\mathcal W_{\widetilde E,\widetilde F}(\tau)
 \cup\M_{\widetilde I\setminus(\widetilde E\cup\widetilde F)}(\tau)\right)\right)=0.
\end{equation} 
Thus, almost every point of $G$ belongs either to $\mathcal W_{\widetilde E,\widetilde F}$ or to $\M_{\widetilde I\setminus(\widetilde E\cup\widetilde F)}$.
\end{definition}

A pair of disjoint sets $(\widetilde E,\widetilde F)$ is called trivial if $|\widetilde E|=0$ or $|\widetilde F|=0$; otherwise, it is called nontrivial.

\begin{definition}[Observable symmetry condition in the Neumann setting]\label{def:osc-n}
Let $T>0$. A measurable set $G\subset[0,T]\times I$ satisfies \emph{(OSC-N)} if, for every nontrivial pair of disjoint measurable sets $(\widetilde E,\widetilde F)$, the set $G$ is not $(\widetilde E,\widetilde F;T)$-observable symmetric.
\end{definition}
\vspace{2mm}

\begin{figure}[htp]
\tikzset{every picture/.style={line width=0.75pt}} 
\begin{tikzpicture}[x=0.75pt,y=0.75pt,yscale=-1,xscale=1,scale=1.34]

\draw  (80,320.17) -- (280,320.17)(99.87,62) -- (99.87,322) (273,315.17) -- (280,320.17) -- (273,325.17) (94.87,69) -- (99.87,62) -- (104.87,69)  ;
\draw  [fill={rgb, 255:red, 212; green, 212; blue, 212 }  ,fill opacity=1 ] (100,90) -- (260,90) -- (260,320) -- (100,320) -- cycle ;
\draw  [dash pattern={on 4.5pt off 4.5pt}]  (100,330) -- (260,330) ;
\draw  [dash pattern={on 4.5pt off 4.5pt}]  (100,340) -- (260,340) ;
\draw  [color={rgb, 255:red, 208; green, 2; blue, 27 }  ,draw opacity=1 ][fill={rgb, 255:red, 208; green, 2; blue, 27 }  ,fill opacity=1 ] (110,338) -- (130,338) -- (130,342) -- (110,342) -- cycle ;
\draw  [color={rgb, 255:red, 208; green, 2; blue, 27 }  ,draw opacity=1 ][fill={rgb, 255:red, 208; green, 2; blue, 27 }  ,fill opacity=1 ] (230,328) -- (250,328) -- (250,332) -- (230,332) -- cycle ;
\draw  [color={rgb, 255:red, 74; green, 144; blue, 226 }  ,draw opacity=1 ][fill={rgb, 255:red, 74; green, 144; blue, 226 }  ,fill opacity=1 ] (190,328) -- (210,328) -- (210,332) -- (190,332) -- cycle ;
\draw  [color={rgb, 255:red, 208; green, 2; blue, 27 }  ,draw opacity=1 ][fill={rgb, 255:red, 255; green, 255; blue, 255 }  ,fill opacity=1 ] (260,190) -- (130,320) -- (110,320) -- (260,170) -- cycle ;
\draw  [color={rgb, 255:red, 74; green, 144; blue, 226 }  ,draw opacity=1 ][fill={rgb, 255:red, 255; green, 255; blue, 255 }  ,fill opacity=1 ] (260,230) -- (170,320) -- (150,320) -- (260,210) -- cycle ;
\draw  [color={rgb, 255:red, 208; green, 2; blue, 27 }  ,draw opacity=1 ][fill={rgb, 255:red, 255; green, 255; blue, 255 }  ,fill opacity=1 ] (260,170) -- (260,190) -- (160,90) -- (180,90) -- cycle ;
\draw  [color={rgb, 255:red, 74; green, 144; blue, 226 }  ,draw opacity=1 ][fill={rgb, 255:red, 255; green, 255; blue, 255 }  ,fill opacity=1 ] (140,90) -- (260,210) -- (260,230) -- (120,90) -- cycle ;
\draw  [color={rgb, 255:red, 74; green, 144; blue, 226 }  ,draw opacity=1 ][fill={rgb, 255:red, 74; green, 144; blue, 226 }  ,fill opacity=1 ] (150,338) -- (170,338) -- (170,342) -- (150,342) -- cycle ;
\draw  [color={rgb, 255:red, 74; green, 144; blue, 226 }  ,draw opacity=1 ][fill={rgb, 255:red, 255; green, 255; blue, 255 }  ,fill opacity=1 ] (100,210) -- (210,320) -- (190,320) -- (100,230) -- cycle ;
\draw  [color={rgb, 255:red, 208; green, 2; blue, 27 }  ,draw opacity=1 ][fill={rgb, 255:red, 255; green, 255; blue, 255 }  ,fill opacity=1 ] (100,170) -- (250,320) -- (230,320) -- (100,190) -- cycle ;
\draw  [color={rgb, 255:red, 208; green, 2; blue, 27 }  ,draw opacity=1 ][fill={rgb, 255:red, 255; green, 255; blue, 255 }  ,fill opacity=1 ] (180,90) -- (200,90) -- (100,190) -- (100,170) -- cycle ;
\draw  [color={rgb, 255:red, 74; green, 144; blue, 226 }  ,draw opacity=1 ][fill={rgb, 255:red, 255; green, 255; blue, 255 }  ,fill opacity=1 ] (240,90) -- (100,230) -- (100,210) -- (220,90) -- cycle ;
\draw  [fill={rgb, 255:red, 126; green, 211; blue, 33 }  ,fill opacity=1 ] (200,110) -- (190,120) -- (200,130) -- (210,120) -- cycle ;
\draw  [fill={rgb, 255:red, 126; green, 211; blue, 33 }  ,fill opacity=1 ] (160,110) -- (150,120) -- (160,130) -- (170,120) -- cycle ;
\draw  [fill={rgb, 255:red, 65; green, 117; blue, 5 }  ,fill opacity=1 ] (120,190) -- (110,200) -- (120,210) -- (130,200) -- cycle ;
\draw  [fill={rgb, 255:red, 65; green, 117; blue, 5 }  ,fill opacity=1 ] (240,190) -- (230,200) -- (240,210) -- (250,200) -- cycle ;
\draw  [fill={rgb, 255:red, 126; green, 211; blue, 33 }  ,fill opacity=1 ] (160,270) -- (150,280) -- (160,290) -- (170,280) -- cycle ;
\draw  [fill={rgb, 255:red, 126; green, 211; blue, 33 }  ,fill opacity=1 ] (200,270) -- (190,280) -- (200,290) -- (210,280) -- cycle ;

\draw (83.56,61.5) node [anchor=north west][inner sep=0.75pt]   [align=left] {$t$};
\draw (286.44,311.19) node [anchor=north west][inner sep=0.75pt]   [align=left] {$x$};
\draw (320,134) node [anchor=north west][inner sep=0.75pt]   [align=left] {$\widetilde E$ is the union of red blocks.};
\draw (320,154) node [anchor=north west][inner sep=0.75pt]   [align=left] {$\widetilde F$ is the union of blue blocks.};
\draw (320,174) node [anchor=north west][inner sep=0.75pt][align=left] {$\mathcal W_{\widetilde E,\widetilde F}(T)$ is the union of \\ dark green and light green parts.};

\draw (320,204) node [anchor=north west][inner sep=0.75pt]{$\M_{\widetilde I\setminus(\widetilde E\cup\widetilde F)}(T)$ is the gray part.};

\draw (320,224) node [anchor=north west][inner sep=0.75pt][align=left]{If $G$ is a subset of $\mathcal W_{\widetilde E,\widetilde F}(T)\cup \M_{\widetilde I\setminus(\widetilde E\cup\widetilde F)}(T)$, \\ $G$ is $(\widetilde E,\widetilde F;T)$-observable symmetric.};
\draw (110,348) node [anchor=north west][inner sep=0.75pt]   [align=left] {The upper black dashed line segment represents $[I;\eta]$;};
\draw (110,363) node [anchor=north west][inner sep=0.75pt]   [align=left] {The lower black dashed line segment represents $[I;\xi]$.};
\end{tikzpicture}
    \caption{OSC in the Neumann setting}
    \label{fig:osc-ne}
\end{figure}
Returning to the counterexample in Section~\ref{subsec:ctex-n}, define
\[ \widetilde E=[I_0;\xi]\cup[I_1;\eta]\subset\widetilde I, \qquad I_0=\left[0,\frac34\right],\quad I_1=\left[0,\frac14\right], \]
and let $\widetilde F=\widetilde I\setminus\widetilde E$. Then $G$ is $(\widetilde E,\widetilde F;1)$-observable symmetric, which explains why observability~\eqref{eq:obs} fails although $G$ satisfies (GCC).

\subsection{Comparison with the Dirichlet case}\label{subsec:diff-n}
The Dirichlet and Neumann symmetry conditions are independent. Their difference is most visible in a boundary triangle. In the Dirichlet setting, two characteristic derivatives retain the same sign after reflection, whereas in the Neumann setting one of them changes sign. Consequently, the boundary-triangle configuration in Figure~\ref{fig:bnd-tri-ne-di} satisfies (OSC-N) but fails (OSC-D). This example also shows why the two geometric conditions should not be identified through the unsupported set assignment $\widetilde A=\widetilde E\cup\widetilde F$.

\begin{figure}[htp]
\tikzset{every picture/.style={line width=0.75pt}} 
\begin{tikzpicture}[x=0.75pt,y=0.75pt,yscale=-1,xscale=1]
\draw   (40,40) -- (200,40) -- (200,200) -- (40,200) -- cycle ;
\draw    (40,200) -- (40,22) ;
\draw [shift={(40,20)}, rotate = 90] [color={rgb, 255:red, 0; green, 0; blue, 0 }  ][line width=0.75]    (10.93,-3.29) .. controls (6.95,-1.4) and (3.31,-0.3) .. (0,0) .. controls (3.31,0.3) and (6.95,1.4) .. (10.93,3.29)   ;
\draw    (40,200) -- (218,200) ;
\draw [shift={(220,200)}, rotate = 180] [color={rgb, 255:red, 0; green, 0; blue, 0 }  ][line width=0.75]    (10.93,-3.29) .. controls (6.95,-1.4) and (3.31,-0.3) .. (0,0) .. controls (3.31,0.3) and (6.95,1.4) .. (10.93,3.29)   ;
\draw  [fill={rgb, 255:red, 126; green, 211; blue, 33 }  ,fill opacity=1 ] (40,40) -- (80,80) -- (40,120) -- cycle ;
\draw  [fill={rgb, 255:red, 126; green, 211; blue, 33 }  ,fill opacity=1 ] (40,120) -- (80,160) -- (40,200) -- cycle ;
\draw  [fill={rgb, 255:red, 126; green, 211; blue, 33 }  ,fill opacity=1 ] (200,40) -- (200,120) -- (160,80) -- cycle ;
\draw  [fill={rgb, 255:red, 126; green, 211; blue, 33 }  ,fill opacity=1 ] (200,120) -- (200,200) -- (160,160) -- cycle ;

\draw (220,120) node [anchor=north west][inner sep=0.75pt]   [align=left] {$G$ is the green part.};
\draw (202,203) node [anchor=north west][inner sep=0.75pt]   [align=left] {$x$};
\draw (42,23) node [anchor=north west][inner sep=0.75pt]   [align=left] {$t$};
\end{tikzpicture}
    \caption{$G$ satisfies (OSC-N) but does not satisfy (OSC-D).}
    \label{fig:bnd-tri-ne-di}
\end{figure}
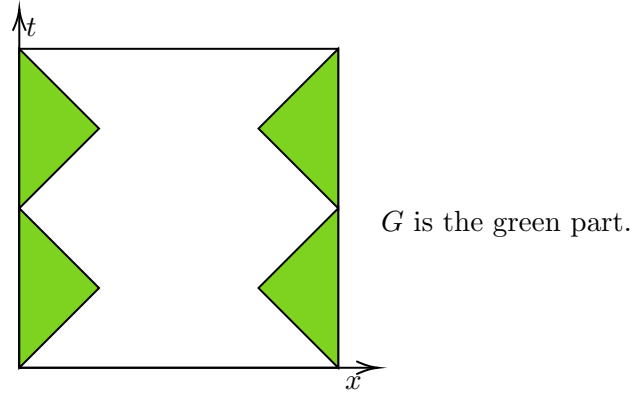

Under the Neumann boundary condition, Lemma~\ref{lem:refl-n} shows that the characteristic derivatives change sign upon reflection. If $G$ contains the boundary triangles shown in green in Figure~\ref{fig:bnd-tri-ne-di}, the identity $\partial_tu=0$ almost everywhere on $G$ forces $\partial_xu=0$ almost everywhere on the corresponding interval, and hence on $I$ for the illustrated configuration.

The converse implication also fails. Let $T=1$ and set
\[
\widetilde E=[I;\xi],\qquad \widetilde F=[I;\eta],\qquad
G=\mathcal W_{\widetilde E,\widetilde F}(1).
\]
This is the union of the two green triangles in Figure~\ref{fig:osc-d-not-n}. Since $\widetilde I\setminus(\widetilde E\cup\widetilde F)=\varnothing$, the set $G$ is $(\widetilde E,\widetilde F;1)$-observable symmetric and hence fails (OSC-N). On the other hand, every pair in $\widetilde E\times\widetilde F$ occurs, up to the negligible boundary ambiguity, as the generalized coordinates of a point of $G$ before time $1$. If $G$ were $(\widetilde A;1)$-observable symmetric, then
\[
\chi_{\widetilde A}(\widetilde y)=\chi_{\widetilde A}(\widetilde z)
\quad\text{for a.e. }(\widetilde y,\widetilde z)\in
\widetilde E\times\widetilde F.
\]
Fubini's theorem would imply that $\chi_{\widetilde A}$ is almost everywhere constant on all of $\widetilde I$, so $\widetilde A$ would be trivial. Thus $G$ satisfies (OSC-D).

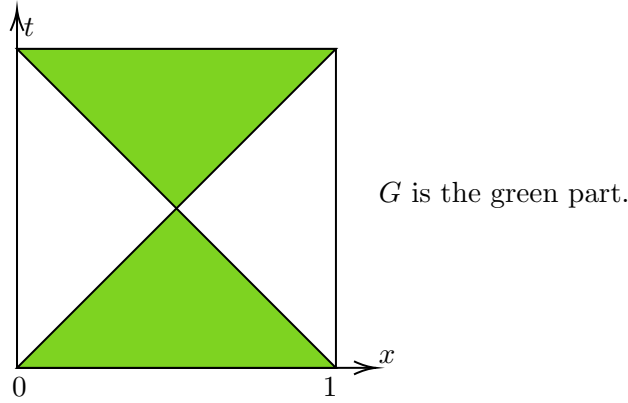
\begin{figure}[htp]
\tikzset{every picture/.style={line width=0.75pt}}
\begin{tikzpicture}[x=0.75pt,y=0.75pt,yscale=-1,xscale=1]
\draw (40,40) -- (200,40) -- (200,200) -- (40,200) -- cycle;
\draw (40,200) -- (40,22);
\draw [shift={(40,20)}, rotate=90] (10.93,-3.29) .. controls (6.95,-1.4) and (3.31,-0.3) .. (0,0) .. controls (3.31,0.3) and (6.95,1.4) .. (10.93,3.29);
\draw (40,200) -- (218,200);
\draw [shift={(220,200)}, rotate=180] (10.93,-3.29) .. controls (6.95,-1.4) and (3.31,-0.3) .. (0,0) .. controls (3.31,0.3) and (6.95,1.4) .. (10.93,3.29);
\draw [fill={rgb,255:red,126;green,211;blue,33},fill opacity=1]
  (40,200) -- (200,200) -- (120,120) -- cycle;
\draw [fill={rgb,255:red,126;green,211;blue,33},fill opacity=1]
  (40,40) -- (200,40) -- (120,120) -- cycle;
\draw (36,203) node [anchor=north west][inner sep=0.75pt] {$0$};
\draw (192,203) node [anchor=north west][inner sep=0.75pt] {$1$};
\draw (220,190) node [anchor=north west][inner sep=0.75pt] {$x$};
\draw (42,23) node [anchor=north west][inner sep=0.75pt] {$t$};
\draw (220,105) node [anchor=north west][inner sep=0.75pt] {$G$ is the green part.};
\end{tikzpicture}
\caption{$G$ satisfies (OSC-D) but does not satisfy (OSC-N).}
\label{fig:osc-d-not-n}
\end{figure}

There is nevertheless a precise relation with the torus condition. Let
\[
\pi(s)=\begin{cases}s,&0\leq s\leq1,\\2-s,&1<s<2,
\end{cases}
\qquad s\in\Ttwo,
\]
and define the reflected periodization of $G$ by
\begin{equation}\label{eq:g-ref}
G_{\mathrm{ref}}=\{(t,s)\in[0,T]\times\Ttwo:(t,\pi(s))\in G\}.
\end{equation}
Thus, on the fundamental domain $[0,2]$, the part over $[1,2]$ is obtained by reflecting $G$ across $x=1$. We use (OSC-T) on $\Ttwo$ in the period-$2$ version of Definition~\ref{def:osc-t}.

\begin{proposition}[Reflection and periodization]\label{prop:osc-ref}
Let $T>0$ and let $G\subset[0,T]\times I$ be measurable. Then $G_{\mathrm{ref}}$ satisfies \refinblack{def:osc-t}{\textnormal{(OSC-T)}} on $[0,T]\times\Ttwo$ if and only if $G$ satisfies both \refinblack{def:osc-d}{\textnormal{(OSC-D)}} and \refinblack{def:osc-n}{\textnormal{(OSC-N)}}.
\end{proposition}

\begin{proof}
Identify the labelled double cover with $\Ttwo$, modulo its endpoints, by
\[
\iota([x;\xi])=x,\qquad \iota([x;\eta])=2-x\pmod 2,
\]
and let $\rho(s)=-s\pmod2$. If $(\widetilde y,\widetilde z)$ are the generalized coordinates of a point of $G$, then the torus characteristic-coordinate pairs of its two lifts in $G_{\mathrm{ref}}$ are
\begin{equation}\label{eq:lifted-coordinates}
(\rho(\iota(\widetilde z)),\iota(\widetilde y))
\quad\text{and}\quad
(\rho(\iota(\widetilde y)),\iota(\widetilde z)),
\end{equation}
up to their order and a null set.

Suppose first that (OSC-D) fails for a nontrivial set $\widetilde A$. Taking
\[
A=\rho(\iota(\widetilde A)),\qquad B=\iota(\widetilde A)
\]
in~\eqref{eq:obs-sym-tor}, formula~\eqref{eq:lifted-coordinates} shows that $G_{\mathrm{ref}}$ is $(A,B)$-observable symmetric. If instead (OSC-N) fails for a nontrivial pair $(\widetilde E,\widetilde F)$, the same conclusion follows by taking
\[
A=\rho(\iota(\widetilde E)),\qquad B=\iota(\widetilde F).
\]
Indeed, the two indicators agree on the lifts of $\mathcal W_{\widetilde E,\widetilde F}$, while both vanish on the lifts of $\M_{\widetilde I\setminus(\widetilde E\cup\widetilde F)}$. In either case the pair $(A,B)$ is nontrivial. Hence (OSC-T) for $G_{\mathrm{ref}}$ implies both interval conditions.

Conversely, suppose that $G_{\mathrm{ref}}$ is $(A,B)$-observable symmetric for a nontrivial pair $(A,B)$. Define two $\{0,1\}$-valued functions on $\widetilde I$ by
\[
c(\widetilde x)=\chi_A(\rho(\iota(\widetilde x))),
\qquad d(\widetilde x)=\chi_B(\iota(\widetilde x)).
\]
By~\eqref{eq:lifted-coordinates}, for almost every point of $G$ with coordinates $(\widetilde y,\widetilde z)$,
\begin{equation}\label{eq:cross-indicators}
c(\widetilde y)=d(\widetilde z),\qquad
c(\widetilde z)=d(\widetilde y).
\end{equation}
Partition $\widetilde I$ into
\[
H_0=\{c=d=0\},\quad H_1=\{c=d=1\},\quad
\widetilde E=\{c=1,d=0\},\quad
\widetilde F=\{c=0,d=1\}.
\]
Equation~\eqref{eq:cross-indicators} says that the two coordinates of a point of $G$ lie either in $H_0\times H_0$, in $H_1\times H_1$, or in $\widetilde E\times\widetilde F$ up to interchange. If both $\widetilde E$ and $\widetilde F$ have positive measure, then $G$ is $(\widetilde E,\widetilde F;T)$-observable symmetric, so (OSC-N) fails. If one of them has positive measure and the other is null, the positive-measure set is missed by the generalized coordinates of $G$ and gives a nontrivial Dirichlet symmetric set; if it has full measure, then $G$ is null and (OSC-D) still clearly fails. Finally, if both are null, nontriviality of $(A,B)$ forces both $H_0$ and $H_1$ to have positive measure, and $G$ is $(H_1;T)$-observable symmetric. Thus failure of (OSC-T) implies failure of at least one of (OSC-D) and (OSC-N), which proves the converse.
\end{proof}

\begin{corollary}[Reflection, GCC, and observability]\label{cor:obs-ref}
Let $T>0$ and let $G\subset[0,T]\times I$ be measurable, and define $G_{\mathrm{ref}}$ by~\eqref{eq:g-ref}. Then
\begin{enumerate}[label=\textnormal{(\roman*)}]
\item $G$ satisfies \refinblack{def:gcc}{\textnormal{(GCC)}} on the interval if and only if $G_{\mathrm{ref}}$ satisfies (GCC) on $\Ttwo$;
\item (OBS-T) holds on $G_{\mathrm{ref}}$ if and only if both (OBS-D) and (OBS-N) hold on $G$.
\end{enumerate}
\end{corollary}

\begin{proof}
Every characteristic on $\Ttwo$ projects under $\pi$ to a generalized reflected characteristic on $I$, and this projection preserves the one-dimensional measure of its intersection with the observation set, apart from the negligible reflection points. This proves (i).

By the torus characterization in~\cite{NiuWangXiang2026}, (OBS-T) on $G_{\mathrm{ref}}$ is equivalent to (GCC) and (OSC-T) on $G_{\mathrm{ref}}$. Part (i), Proposition~\ref{prop:osc-ref}, and Theorem~\ref{thm:main-obs} identify these two conditions with (GCC), (OSC-D), and (OSC-N) on $G$, which in turn are equivalent to the simultaneous validity of (OBS-D) and (OBS-N).
\end{proof}

In short, the generalized coordinates associated with points of $G$ must come from the same part of $\widetilde I$ in a Dirichlet symmetric configuration, but from two different parts in a Neumann symmetric configuration. Proposition~\ref{prop:osc-ref} shows that the reflected torus condition excludes exactly both possibilities.

\section{Unique continuation in the Neumann setting}\label{sec:ucp-n}

We now study the unique continuation property (UCP) in the Neumann setting. The proof follows the strategy used in Section~\ref{sec:ucp-d}, but the sign change in Lemma~\ref{lem:refl-n} requires a different correspondence between functions on $\widetilde I$ and the initial characteristic derivatives.

\begin{proposition}\label{prop:ucp-osc-n}
Let $T>0$, and let $G\subset [0,T]\times I$ be measurable. Assume that every solution $u$ of~\eqref{eq:wave-ne} satisfies
\[ \partial_t u=0 \quad\text{a.e. on }G \quad\Longrightarrow\quad u\text{ is constant}. \]
Then $G$ satisfies \refinblack{def:osc-n}{\textnormal{(OSC-N)}}.
\end{proposition}

\begin{proof}
We argue by contradiction. Suppose that there is a nontrivial pair of disjoint measurable sets $\widetilde E,\widetilde F\subset\widetilde I$ such that $G$ is $(\widetilde E,\widetilde F;T)$-observable symmetric. Thus,
\begin{equation}\label{eq:ucp-sym-ne}
G\subset \mathcal W_{\widetilde E,\widetilde F}(T) \cup \M_{\widetilde I\setminus(\widetilde E\cup\widetilde F)}(T) \quad\text{up to a null set}.
\end{equation}

Define the simple function
\[ f=\chi_{\widetilde E}-\chi_{\widetilde F} \quad\text{on }\widetilde I. \]
Choose initial data $(u_0,u_1)\in H^1(I)\times L^2(I)$ such that
\begin{equation}\label{eq:init-cor-ne} \left.\partial_\xi u(x)\right|_{t=0}=f([x;\eta]),
\qquad
  \left.\partial_\eta u(x)\right|_{t=0}=-f([x;\xi]).
\end{equation}
Lemma~\ref{lem:refl-n} and~\eqref{eq:ucp-sym-ne} then give $\partial_tu=0$ almost everywhere on $G$. The resulting solution is nonconstant, contradicting the assumed UCP.
\end{proof}

Motivated by~\eqref{eq:init-cor-ne}, given a function $f$ on $\widetilde I$, define a pair $(g,h)$ on $I$ by
\[ g(x)=f([x;\eta]), \qquad h(x)=-f([x;\xi]), \qquad x\in I. \]
Conversely, a pair $(g,h)$ on $I$ determines $f$ on $\widetilde I$ through these identities. We say that $f$ \emph{corresponds to} $(g,h)$ in the Neumann setting.

\begin{proposition}\label{prop:ucp-n-sim}
Let $T>0$, and let $G\subset [0,T]\times I$ satisfy \refinblack{def:w-gcc}{\textnormal{(Weak GCC)}}. Suppose that $u$ solves~\eqref{eq:wave-ne}, that $\partial_tu=0$ almost everywhere on $G$, and that $(\partial_\xi u,\partial_\eta u)|_{t=0}$ corresponds to a function $f$ on $\widetilde I$ in the Neumann setting. Then
\begin{equation}\label{eq:coun-sim-ne} f(\widetilde x)=\sum_{k\in\Lambda}s_k
\left(\chi_{\widetilde E_k}(\widetilde x)-\chi_{\widetilde F_k}(\widetilde x)\right)
  \quad\text{for a.e. }\widetilde x\in\widetilde I,
\end{equation}
where $\Lambda$ is at most countable, the numbers $s_k>0$ are pairwise distinct, and the sets
\[ \widetilde E_k=f^{-1}(s_k), \qquad \widetilde F_k=f^{-1}(-s_k) \]
have positive measure. The sets $\widetilde E_k$ and $\widetilde F_k$ are pairwise disjoint and mutually disjoint. We allow $\Lambda$ to be empty, in which case $f=0$ almost everywhere.
\end{proposition}

\begin{proof}
There is a set $\mathcal N_1\subset G$ such that $\meas_{\R^2}(\mathcal N_1)=0$ and $\partial_tu=0$ on $G\setminus\mathcal N_1$. Let $\mathcal N_2$ be the null set in Lemma~\ref{lem:refl-n}, and set $\mathcal N=\mathcal N_1\cup\mathcal N_2$. For $\widetilde x\in\widetilde I$, define
\[ J_{\widetilde x} =\{s\in[0,T]:L_{\widetilde x}(s)\in G\setminus\mathcal N\}. \]

By \refinblack{def:w-gcc}{\textnormal{(Weak GCC)}}, for almost every $\widetilde x\in\widetilde I$,
\[ \meas_{\R}(J_{\widetilde x})>0. \]
Denote the set of such $\widetilde x$ by $\mathcal P$. For $\widetilde x\in\mathcal P$, let
\[ E_{\widetilde x} =\{\widetilde y\in\widetilde I: \widetilde y=F(\widetilde x,s)\text{ for some }s\in J_{\widetilde x}\}. \]
Lemma~\ref{lem:refl-n} implies
$E_{\widetilde x}\subset f^{-1}(-f(\widetilde x))$.

For $\widetilde x\in\mathcal P$ and $\widetilde y\in E_{\widetilde x}$, there are at most $N_T=\lfloor T\rfloor+2$ values of $s\in J_{\widetilde x}$ such that $F(\widetilde x,s)=\widetilde y$. Hence, for almost every $\widetilde x\in\widetilde I$,
\[\meas_{\R}\!\left(f^{-1}(-f(\widetilde x))\right)\geq \meas_{\R}(E_{\widetilde x})
\geq \frac{1}{N_T}\meas_{\R}(J_{\widetilde x})>0.
\]

The set of nonzero values $a$ for which $f^{-1}(a)$ has positive measure is at most countable. The preceding estimate shows that these values occur in pairs $\{s_k,-s_k\}$ with $s_k>0$. Enumerating their distinct absolute values by $k\in\Lambda$ gives~\eqref{eq:coun-sim-ne}.
\end{proof}

\begin{proposition}\label{prop:ucp-n}
Let $T>0$, and let $G\subset[0,T]\times I$ satisfy \refinblack{def:w-gcc}{\textnormal{(Weak GCC)}} and \refinblack{def:osc-n}{\textnormal{(OSC-N)}}. If $u$ solves~\eqref{eq:wave-ne} and $\partial_tu=0$ almost everywhere on $G$, then $u$ is constant.
\end{proposition}

\begin{proof}
By Proposition~\ref{prop:ucp-n-sim}, the function $f$ corresponding to $(\partial_\xi u,\partial_\eta u)|_{t=0}$ has the form~\eqref{eq:coun-sim-ne}.

Suppose that $\Lambda$ is nonempty. Set
\[ \widetilde E=\bigcup_{k\in\Lambda}\widetilde E_k, \qquad \widetilde F=\bigcup_{k\in\Lambda}\widetilde F_k. \]
The pair $(\widetilde E,\widetilde F)$ is nontrivial. If
\[ \mathcal S=\{(t,x)\in[0,T]\times I:\partial_tu(t,x)=0\}, \]
then Lemma~\ref{lem:refl-n} gives
\[\mathcal S\subset\mathcal W_{\widetilde E,\widetilde F}(T)
\cup\M_{\widetilde I\setminus(\widetilde E\cup\widetilde F)}(T)\quad\text{up to a null set}.
\]
Since $G\subset\mathcal S$ up to a null set, $G$ is $(\widetilde E,\widetilde F;T)$-observable symmetric, contradicting (OSC-N). Therefore $\Lambda$ is empty, so $f=0$ almost everywhere and $u$ is constant.
\end{proof}

\section{Observability}\label{sec:obs}

Having established UCP in Sections~\ref{sec:ucp-d} and~\ref{sec:ucp-n}, we now combine it with weak observability to prove Theorem~\ref{thm:main-obs}. The weak observability estimates are obtained by extending the interval problems to the period-$2$ torus and applying the estimates in~\cite{NiuWangXiang2026}.

\subsection{Dirichlet case}

We first record the odd-extension argument.

\begin{lemma}\label{lem:odd-ext}
Let $T>0$, and let $u$ be a solution of~\eqref{eq:wave-di}. Extend the initial data $(u_0,u_1)$ oddly to $[-1,1]$ by
\[v_0(x)=\begin{cases}
u_0(x),&0\leq x\leq1,\\
-u_0(-x),&-1\leq x\leq0,
\end{cases}\qquad v_1(x)=\begin{cases}
u_1(x),&0\leq x\leq1,\\
-u_1(-x),&-1\leq x\leq0.
\end{cases}
\]
Extend these functions periodically with period $2$ to the torus
$\Ttwo=\R/(2\Z)$. Let $v$ solve
\begin{equation}\label{eq:wave-tor-2}\begin{cases}
(\partial_t^2-\partial_x^2)v=0,\\
(v,\partial_tv)|_{t=0}=(v_0,v_1),\\
v_0\in\dot H^1(\Ttwo),\quad v_1\in L^2(\Ttwo).
\end{cases}\end{equation}
Then, for every $t\in[0,T]$, one has $v(t,x)=-v(t,-x)$ for almost every $x\in\Ttwo$, and $u(t,x)=v(t,x)$ for almost every $x\in I$.
\end{lemma}

\begin{proof}
Write the Dirichlet initial data as sine series:
\[ u_0(x)=\sum_{k=1}^{\infty}a_k\sin(k\pi x), \qquad u_1(x)=\sum_{k=1}^{\infty}b_k\sin(k\pi x) \quad\text{for a.e. }x\in I. \]
Their odd periodic extensions have Fourier expansions
\[v_0(x)=\sum_{k\in\Z\setminus\{0\}}\widetilde a_k e^{ik\pi x}, \qquad
v_1(x)=\sum_{k\in\Z\setminus\{0\}}\widetilde b_k e^{ik\pi x},
\]
where, for $k\geq1$,
\[\widetilde a_k=\frac{a_k}{2i}, \quad \widetilde a_{-k}=-\frac{a_k}{2i},
\qquad
\widetilde b_k=\frac{b_k}{2i}, \quad \widetilde b_{-k}=-\frac{b_k}{2i}.
\]
The torus solution is
\[v(t,x)=\sum_{k\in\Z\setminus\{0\}}
\left(\widetilde a_k\cos(k\pi t)+\frac{\widetilde b_k}{k\pi}\sin(k\pi t)\right)e^{ik\pi x},
\]
whereas the Dirichlet solution is
\[ u(t,x)=\sum_{k=1}^{\infty} \left( a_k\cos(k\pi t) +\frac{b_k}{k\pi}\sin(k\pi t) \right)\sin(k\pi x). \]
The stated identities follow directly from these expansions.
\end{proof}

\begin{proposition}\label{prop:w-obs-d}
Let $T>0$, and let $G\subset[0,T]\times I$ be measurable and satisfy \refinblack{def:gcc}{\textnormal{(GCC)}}. There is a constant $C>0$ such that
\begin{equation}\label{eq:w-obs-di} \|\partial_xu_0\|_{L^2(I)}^2+\|u_1\|_{L^2(I)}^2 \leq C\left(\iint_G|\partial_tu(t,x)|^2\,\dd x\,\dd t +\|u_0\|_{\dot L^2(I)}^2+\|u_1\|_{H^{-1}(I)}^2\right)
\end{equation}
for every solution $u$ of~\eqref{eq:wave-di} with
$(u_0,u_1)\in H_0^1(I)\times L^2(I)$.
\end{proposition}

\begin{proof}
Perform the odd extension of Lemma~\ref{lem:odd-ext}, and continue to denote the extended data by $(v_0,v_1)$ and the corresponding solution by $v$. Let $E$ be the union of $G$ and its reflection across $x=0$. If $G$ satisfies (GCC) in the Dirichlet setting, then $E$ satisfies (GCC) on $\Ttwo$. The weak observability estimate in~\cite{NiuWangXiang2026} gives
\begin{equation}\label{eq:w-obs-tor} \|\partial_xv_0\|_{L^2(\Ttwo)}^2+\|v_1\|_{L^2(\Ttwo)}^2
\leq C\left(\iint_E|\partial_tv(t,x)|^2\,\dd x\,\dd t
+\|v_0\|_{\dot L^2(\Ttwo)}^2+\|v_1\|_{H^{-1}(\Ttwo)}^2\right).
\end{equation}
By odd symmetry, both sides of~\eqref{eq:w-obs-tor} are twice the corresponding sides of~\eqref{eq:w-obs-di}.
\end{proof}

Combining Proposition~\ref{prop:ucp-d} with Proposition~\ref{prop:w-obs-d}, the standard compactness--uniqueness argument of Bardos--Lebeau--Rauch~\cite{BLR-gcc} yields observability.

\begin{proposition}\label{prop:osc-gcc-obs-d}
Let $T>0$, and let $G\subset[0,T]\times I$ be measurable. If $G$ satisfies \refinblack{def:gcc}{\textnormal{(GCC)}} and \refinblack{def:osc-d}{\textnormal{(OSC-D)}}, then~\eqref{eq:obs} holds on $G$ for every solution of~\eqref{eq:wave-di}.
\end{proposition}

\begin{proof}[Proof of the Dirichlet part of Theorem~\ref{thm:main-obs}]
The sufficiency follows from Proposition~\ref{prop:osc-gcc-obs-d}. Proposition~\ref{prop:ucp-osc-d} gives the necessity of (OSC-D), while the necessity of (GCC) follows from the standard geometric-optics argument; see~\cite{NiuWangXiang2026}.
\end{proof}

\subsection{Neumann case}

Under Neumann boundary conditions, the same extension argument gives the corresponding weak observability estimate. We state the result; its proof is analogous to the Dirichlet case.

\begin{proposition}\label{prop:w-obs-n}
Let $T>0$, and let $G\subset[0,T]\times I$ be measurable and satisfy \refinblack{def:gcc}{\textnormal{(GCC)}}. There is a constant $C>0$ such that
\begin{equation}\label{eq:w-obs-ne} \|\partial_xu_0\|_{L^2(I)}^2+\|u_1\|_{L^2(I)}^2
\leq C\left(\iint_G|\partial_tu(t,x)|^2\,\dd x\,\dd t
+\|u_0\|_{\dot L^2(I)}^2+\|u_1\|_{H^{-1}(I)}^2\right)
\end{equation}
for every solution $u$ of~\eqref{eq:wave-ne} with $(u_0,u_1)\in\dot H^1\times L^2$.
\end{proposition}

\begin{proposition}\label{prop:osc-gcc-obs-n}
Let $T>0$, and let $G\subset[0,T]\times I$ be measurable. If $G$ satisfies \refinblack{def:gcc}{\textnormal{(GCC)}} and \refinblack{def:osc-n}{\textnormal{(OSC-N)}}, then~\eqref{eq:obs} holds on $G$ for every solution of~\eqref{eq:wave-ne}.
\end{proposition}

\begin{proof}[Proof of the Neumann part of Theorem~\ref{thm:main-obs}]
Sufficiency is Proposition~\ref{prop:osc-gcc-obs-n}. The necessity of (OSC-N) follows from Proposition~\ref{prop:ucp-osc-n}, while (GCC) is necessary by the same argument as in the Dirichlet case.
\end{proof}

\begin{proof}[Proof of Theorem~\ref{thm:main-ucp}]
The necessity of (Weak GCC) follows from Proposition~\ref{prop:ucp-wgcc}. Combine Propositions~\ref{prop:ucp-osc-d} and~\ref{prop:ucp-d} in the Dirichlet setting, and Propositions~\ref{prop:ucp-osc-n} and~\ref{prop:ucp-n} in the Neumann setting.
\end{proof}
\vspace{2mm}

\noindent\textbf{Acknowledgement}\; Shengquan Xiang is partially supported by NSFC 12571474 and Shanghai Qiguang Natural Science Development Foundation.

\bibliographystyle{abbrv}
\IfFileExists{mainbib.bib}{\bibliography{mainbib}}{\bibliography{../DW1,additional}}

\end{document}